\documentclass[]{amsart}
\usepackage{amsmath, amssymb, amsthm, mathtools}

\usepackage[scr=boondox]{mathalfa}

\usepackage{graphicx}
\usepackage{svg}
\svgpath{{img/}}

\usepackage{relsize}

\usepackage{caption, subcaption}

\usepackage{tikz,tikz-cd}
\usepackage {pgfplots}
\pgfplotsset{compat=1.17}
\usepackage{hyperref}

\usepackage{comment}
\usepackage{cleveref}

\usepackage{thmtools}

\declaretheoremstyle[
  spaceabove=1em, spacebelow=1em,
  headfont=\bfseries,
  notefont=\bfseries, notebraces={(}{)},
  bodyfont=\itshape,
  postheadspace=0.5em,
]{thmlike}

\declaretheoremstyle[
  spaceabove=1em, spacebelow=1em,
  headfont=\normalfont,
  notefont=\normalfont, notebraces={(}{)},
  bodyfont=\normalfont,
  postheadspace=1em,
  qed=$\blacksquare$
]{prooflike}

\declaretheoremstyle[
  spaceabove=1em, spacebelow=1em,
  headfont=\bfseries,
  notefont=\bfseries, notebraces={(}{)},
  bodyfont=\normalfont,
  postheadspace=0.5em,
]{deflike}

\declaretheorem[style=thmlike, numberwithin=section]{theorem}
\declaretheorem[style=thmlike, sibling=theorem]{corollary}
\declaretheorem[style=deflike, sibling=theorem]{definition}

\declaretheorem[style=thmlike, sibling=theorem]{proposition}
\declaretheorem[style=deflike, sibling=theorem]{example}

\declaretheorem[style=deflike, sibling=theorem]{remark}

\usepackage{cleveref}

\newcommand*\diff{\mathop{}\!\mathrm{d}}

\newcommand{\bi}[2][]{{}^{b^{#1}} #2}

\providecommand\given{}
\newcommand\SetSymbol[1][]{%
\nonscript\:#1\vert
\allowbreak
\nonscript\:
\mathopen{}}
\DeclarePairedDelimiterX\Set[1]\{\}{%
\renewcommand\given{\SetSymbol[\delimsize]}
#1
}

\newcommand{\eva}[1]{#1}

\makeatletter
 \def\@textbottom{\vskip \z@ \@plus 1pt}
 \let\@texttop\relax
\makeatother

\tikzset{>=latex} 
\tikzset{declare function={
  penrose(\x,\c)  = {\fpeval{2/pi*atan( (sqrt((1+tan(\x)^2)^2+4*\c*\c*tan(\x)^2)-1-tan(\x)^2) /(2*\c*tan(\x)^2) )}};
  penroseu(\x,\t) = {\fpeval{atan(\x+\t)/pi+atan(\x-\t)/pi}};
  penrosev(\x,\t) = {\fpeval{atan(\x+\t)/pi-atan(\x-\t)/pi}};
  kruskal(\x,\c)  = {\fpeval{asin( \c*sin(2*\x) )*2/pi}};
}}

\title{Which singular tangent bundles are isomorphic?}
\author{Eva Miranda} 
\address{Eva Miranda,
Laboratory of Geometry and Dynamical Systems and SYMCREA research unit, Department of Mathematics, Universitat Polit\`{e}cnica de Catalunya, Barcelona \& Institute of Mathematics of UPC, IMTech \& Centre de Recerca Matemàtica, CRM \& Forschungsinstitut für Mathematik, ETH Zürich }
\thanks{Eva Miranda is supported by the Catalan Institution for Research and Advanced Studies via an ICREA Academia Prize 2021 and by the Alexander Von Humboldt Foundation via a Friedrich Wilhelm Bessel Research Award. The Spanish State Research Agency supports both authors through the Severo Ochoa and Mar\'{\i}a de Maeztu Program for Centers and Units of Excellence in R\&D (project CEX2020-001084-M) and the project reference  PID2023-146936NB-I00 funded by MICIU/AEI/
10.13039/501100011033 and, by ERDF/EU}
\author{Pablo Nicolás}
\address{Pablo Nicolás,
 Centre de Recerca Matemàtica, CRM \& Laboratory of Geometry and Dynamical Systems and SYMCREA research unit, Department of Mathematics, Universitat Polit\`{e}cnica de Catalunya, Barcelona}
\thanks{Pablo Nicolás is supported under a MDM-FPI contract with reference code PRE2022-102974.}
\date{}

\subjclass{55R15 (primary), 53D17, 55R50 (secondary)}

\begin{document}

\begin{abstract}

Logarithmic and \( b \)-tangent bundles provide a versatile framework for addressing singularities in geometry. Introduced by Deligne and Melrose, these \emph{modified bundles} resolve singularities by reframing singular vector fields as well-behaved sections of these \emph{singular bundles}. This approach has gained significant attention in symplectic geometry, particularly through its applications to the study of Poisson manifolds that are symplectic away from a hypersurface (\( b^m \)-symplectic forms).

In this article, we investigate the conditions under which these singular tangent bundles are isomorphic to the tangent bundle or other singular bundles, analyzing in detail the low-dimensional case and the case of spheres. We also examine the existence of geometric structures in light of these conditions.  Furthermore, we establish a Poincaré–Hopf theorem for the \( b^m \)-tangent bundle, offering new insights into the interplay between singular structures and topological invariants.

\end{abstract}

\maketitle

\section{Introduction}

The concept of logarithmic tangent bundles emerged in the 1970s, primarily through the seminal work of Pierre Deligne. In \cite{deligne_equations_1970}, Deligne introduced logarithmic differential forms as a means of studying complex manifolds with divisors that exhibit singularities. This elegant framework enabled the analysis of differential forms with controlled singularities along a divisor, culminating in the definition of the logarithmic tangent bundle as the dual of the sheaf of logarithmic differential forms.

A similar line of reasoning was later pursued by Melrose \cite{melrose_transformation}, who developed a microlocal framework for pseudodifferential and Fourier integral operators on manifolds with boundary, introducing key geometric ideas that later crystallized in the $b$-calculus. These methods were subsequently used by Melrose to give a conceptual proof of the Atiyah–Patodi–Singer index theorem for smooth manifolds with boundary \cite{melrose_atiyah-patodi-singer_1993}. In this context, the sections of the $b$-tangent bundle consist of smooth vector fields that are tangent to the boundary. These sections form a $\mathcal{C}^\infty(M)$-module, which is both freely and finitely generated. Consequently, there exists a bundle whose sections are precisely such vector fields, now widely known as the $b$-tangent bundle (where ``$b$'' stands for ``boundary'').

The utility of these bundles lies in their ability to address singular objects by reframing them as sections of the \emph{modified bundle}, or its dual, where the singularity is effectively resolved. That is, the objects are no longer singular when considered as sections of the modified bundle. This perspective allows the treatment of various singular phenomena in a natural manner and has recently garnered significant interest, particularly in symplectic geometry. 

$b$-Symplectic forms have emerged as essential tools for understanding Poisson structures that are symplectic away from a hypersurface \cite{guillemin_symplectic_2014, gualtieri_symplectic_2014}. These forms provide a natural framework for modeling systems with singularities, such as problems in celestial mechanics \cite{kiesenhofermiranda, delshamskiesenhofermiranda} or the quantization of symplectic manifolds with boundary \cite{nest_formal_1996}. The desingularization procedure in \cite{guillemin_desingularizing_2019} connects $b^m$-symplectic geometry to classical symplectic geometry.

Another class of singular bundles is given by the \emph{edge tangent bundle}. Following the introduction of the $b$-tangent bundle by Melrose, and the subsequent development of the elliptic theory of totally characteristic operators with Mendoza~\cite{MeMe}, edge structures were developed by Mazzeo~\cite{mazzeo_elliptic} in his study of elliptic operators on manifolds with fibred boundary, marking the first systematic appearance of edge geometry. Edge structures provide a natural framework for the analysis of singularities arising in geometric and physical contexts. They are associated with submanifolds equipped with a fibration, and the corresponding edge vector fields are tangent both to the submanifold and to the fibres of the fibration. A notable application arises in the twistor space of $\mathbb{H}^4$, where the Eells--Salamon almost complex structure becomes singular near the boundary (see also~\cite{fine_knots_2022}). Edge geometry provides a smooth framework to address such singular behaviour.

In this article, we analyze the conditions under which singular tangent bundles, such as the \(b^m\)-tangent bundle, are isomorphic either to the standard tangent bundle or to other singular tangent bundles. Such distinctions also shed light on the obstructions for a manifold to admit certain geometric structures, namely symplectic,  almost symplectic, complex or contact, since these structures can be interpreted as sections of appropriate exterior powers of the cotangent bundle. While the obstructions to the existence of symplectic and contact structures are subtle and highly intricate, the situation simplifies considerably for almost symplectic and almost contact structures as the obstructions are given by the reduction of the corresponding structural group. An almost symplectic manifold has vanishing odd Stiefel–Whitney classes; hence this vanishing is a necessary condition for the existence of an almost symplectic structure. In the almost contact case, one must additionally require orientability, that is, the vanishing of the first Stiefel--Whitney class \(\mathrm{w}_1(\mathrm{T}M)\), together with the vanishing of all higher odd Stiefel--Whitney classes (see, for instance, \cite{Thomas1967,Geiges2008,husemoller-fibrebundles,libermann-complex-1955}). Stiefel–Whitney classes also control the existence of spin structures, which play a central role in geometric quantization and index-theoretic approaches to quantization. In this article, we focus specifically on these characteristic classes. Furthermore, we establish isomorphisms between the odd and even \(b^m\)-tangent bundles for manifolds of the same parity. This result implies that the obstructions to the existence of \(b^m\)-symplectic or \(b^m\)-contact structures are determined entirely by the parity of \(m\).

We focus on the associated characteristic classes and present a direct proof that \(b^{2m}\)-symplectic manifolds are symplectic, bypassing the desingularization approach. Furthermore, we examine the Euler class of the \(b^m\)-tangent bundle and establish a Poincaré–Hopf theorem for \(b^m\)-manifolds. This result reveals a fundamental connection between the topology of the pair \((M,Z)\) and the dynamics of \(b\)-vector fields, highlighting their deep interplay.

A novel aspect of our work is the connection between these geometric structures and combinatorial problems in graph theory, which provides a classification perspective on the phenomena described above. Rather than addressing only the existence of singular geometric structures, we show how the topology of the underlying manifold and the singular hypersurface constraints and organizes the possible isomorphism types of singular tangent bundles. To illustrate this viewpoint, we present a detailed analysis of the case of spheres, where these classification mechanisms can be made completely explicit. Finally, we show how this approach extends naturally to edge structures, further highlighting the unifying role of topology in the classification of singular geometric structures.

\textbf{Acknowledgements:} We thank Kai Cieliebak, Viktor Ginzburg, and Yael Karshon for raising the question of whether there exists an isomorphism between the $b$-tangent bundle and the tangent bundle on several occasions. Their pertinent observations were instrumental in the development of this article. We would also like to thank the referees for their suggestions, which greatly improved the presentation of \cref{sec:Euler-class,sec:iso-b-spheres} and corrected a mistake in a former version of Theorem \Cref{thm:criterion-iso-bspheres}. \eva{We thank Robert Cardona for his interest in this paper}.

\section{Preliminaries} \label{sec:prelim}

In this section, we provide an overview of the key theoretical tools used throughout the paper. We cover classical theorems and introduce necessary definitions regarding $E$-bundles. We conclude by presenting a general overview of the procedure we will follow in the upcoming sections.

\subsection{Construction of vector bundles from gluing data} \label{subsec:construction-gluing-data}

It is well known that equivalence classes of vector bundles over a smooth manifold $M$ can be reformulated in terms of gluing data in a suitable open cover $\mathcal{U}$ of $M$. This construction is often known as the \emph{vector bundle construction theorem}. If the bundle is trivializable in all open sets $U_\alpha \in \mathcal{U}$, the gluing data amounts to a collection of functions $g_{\alpha \beta} \colon U_\alpha \cap U_\beta \to \mathrm{GL}_n(\mathbb{R})$ satisfying the conditions $g_{\alpha \beta} g_{\beta \gamma} = g_{\alpha \gamma}$ and $g_{\alpha \alpha} = \operatorname{id}$, whenever these expressions are defined. The reader is referred to \cite[Theorem 3.2]{steenrod_topology_1951} or \cite[Proposition 5.2]{kobayashi_foundations_1963} (for principal bundles) for further details.

We recall this correspondence here in the more general case when the vector bundle $\pi \colon E \to M$ is not trivializable in the open sets $U_\alpha$ of an open cover $\mathcal{U}$. We denote by $i_\alpha \colon U_\alpha \to M$ and $i_{\alpha \beta}^\alpha \colon U_\alpha \cap U_\beta \to U_\alpha$ the canonical inclusions of sets. In this case, the data for the construction of a vector bundle is a collection of bundles $\pi_\alpha \colon E_\alpha \to U_\alpha$. The \emph{gluing data} is a collection of vector bundle morphisms $g_{\beta \alpha} \colon (i_{\alpha \beta}^\alpha)^* E_\alpha \to (i_{\alpha \beta}^\beta)^* E_\beta$ satisfying the conditions $g_{\alpha \beta} g_{\beta \gamma} = g_{\alpha \gamma}$ and $g_{\alpha \alpha} = \operatorname{id}$. Given such data, we can prove that the space
\begin{equation*}
    E = \bigsqcup_{U_\alpha \in \mathcal{U}} E_\alpha \bigg/ \sim,
\end{equation*}
where we say $v \in E_\alpha$ is related to $u \in E_\beta$ if and only if $u = g_{\beta \alpha}(v)$, is a smooth vector bundle over $M$. The conditions imposed on the functions $g_{\alpha \beta}$ ensure $\sim$ is an equivalence relation. We can recover the gluing data of a vector bundle $\pi \colon E \to M$ in an open cover $\mathcal{U} = \{U_\alpha\}$ by setting $E_\alpha \coloneqq (i_\alpha)^* E$ and $g_{\alpha \beta} \coloneqq (i_{\alpha \beta})^* \operatorname{id}$.

Let us assume that we have an isomorphism $f \colon E \to E'$ between vector bundles constructed over the \emph{same} open cover $\mathcal{U}$ and with the same local models $\pi_\alpha \colon E_\alpha \to U_\alpha$ with gluing data $g_{\alpha \beta}$ and $g_{\alpha \beta}'$, respectively. The local isomorphisms $f_\alpha \coloneqq E_\alpha \to E_\alpha$ glue to the global function $f \colon E \to E'$ if and only if the condition $f_\beta g_{\beta \alpha} = g_{\beta\alpha}' f_\alpha$ is satisfied for all pairs $U_\alpha, U_\beta \in \mathcal{U}$. This condition gives an equivalence relation on the collection of gluing data $\{g_{\alpha \beta}\}$, which is a one-to-one correspondence with the equivalence classes of vector bundles under isomorphism.

We finally recall that the automorphisms $f \colon E \to E$ of a vector bundle $\pi \colon E \to M$ correspond to sections of a $\mathrm{GL}_n(\mathbb{R})$-bundle called the bundle of \emph{gauge transformations} of $E$ and written $\mathcal{G}(E)$. Sections of $\mathcal{G}(E)$ are accordingly referred to as gauge transformations. The bundle $\mathcal{G}(E)$ can be abstractly defined as the associated bundle $\mathcal{G}(E) = \mathcal{F}(E) \times_{\mathrm{GL}_n(\mathbb{R})} \mathrm{GL}_n(\mathbb{R})$, where $\mathcal{F}(E)$ is the frame bundle of $E$ and the action of $\mathrm{GL}_n (\mathbb{R})$ on $\mathrm{GL}_n (\mathbb{R})$ is the adjoint or conjugation action. Because the determinant function $\det \colon \mathrm{GL}_n (\mathbb{R}) \to \mathbb{R}^*$ is invariant by conjugation, we observe that the determinant of a gauge transformation is well-defined.

\subsection{The Serre-Swan theorem and \texorpdfstring{$E$}{E}-bundles} \label{subsec:Serre-Swan}

The construction of vector bundles in terms of local gluing data can be used to recover a vector bundle from its space of sections. Fix a vector bundle $\pi \colon E \to M$. In an open set $U \subseteq M$ where we have a trivialization $\varphi \colon E|_U \to U \times \mathbb{R}^n$, the space of sections $\Gamma_U E$ is a free module over $\mathcal{C}^\infty(U)$ with a basis given by the sections $\varphi^{-1}(e_i)$.

It turns out that the converse is true: every locally free sheaf of $\mathcal{C}^\infty$-modules can be realized as the sheaf of sections of an appropriate vector bundle $E$. Here, by locally free sheaf we mean a sheaf $\mathcal{E}$ over $M$ such that for every point $p \in M$ there exists an open neighbourhood $U \subseteq M$ such that $\mathcal{E}(U) \cong \mathcal{C}^\infty(U)$ \emph{as sheaves}. The last remark is important, as we are saying that the generators of $\mathcal{E}(V)$ for any subset $V \subseteq U$ are obtained from those of $\mathcal{E}(U)$. This is the content of a theorem by Serre, announced originally in the context of algebraic varieties. In the following statement, the sheaf $\mathcal{S}(E)$ denotes the sheaf of germs of sections of a vector bundle $E$, and the sheaf $\mathcal{O}$ is the structure sheaf of the base space.

\begin{theorem}[{Serre \cite[\textbf{Proposition 5 \S 41}]{serre_faisceaux_1955}}] \label{thm:serre-bundle}
    Assume $E$ is an algebraic vector bundle of rank $r$. The sheaf $\mathcal{S}(E)$ is locally isomorphic to $\mathcal{O}_V^r$; in particular, it is an algebraic coherent sheaf. Conversely, every sheaf $\mathcal{F}$ locally isomorphic to $\mathcal{O}_V^r$ is isomorphic to $\mathcal{S}(E)$ for some algebraic vector bundle $E$.
\end{theorem}

We present a sketch of the construction of a bundle associated with a locally free sheaf $\mathcal{E}$. In what follows, we denote the sheaf by $\mathcal{E}$ and take the structure sheaf $\mathcal{O}$ to be the sheaf of smooth functions $\mathcal{C}^\infty$. We consider the collection of open sets $\mathcal{U} = \{ U_\alpha \}$ for which the sheaf $\mathcal{E}$ is free, that is, $\mathcal{E} (U_\alpha) \cong (\mathcal{C}^\infty(U_\alpha))^n$. We can construct, at least locally, vector bundles $\pi_\alpha \colon E_\alpha \to U_\alpha$ such that $\Gamma E_\alpha = \mathcal{E}(U_\alpha)$. Indeed, simply take $E_\alpha = U_\alpha \times \mathbb{R}^n$, where generators of the fibre $\mathbb{R}^n$ are in correspondence with elements in the canonical basis of $\mathcal{E}(U_\alpha) = (\mathcal{C}^\infty(U_\alpha))^n$.

To construct the global bundle $\pi \colon E \to M$ representing $\mathcal{E}$ we use the bundle construction theorem in \cref{subsec:construction-gluing-data}. The gluing data is obtained from the local triviality condition of $\mathcal{E}$ as follows. Given two sets $U_\alpha$ and $U_\beta$, the bases $e_1, \ldots, e_n$ of $\mathcal{E}(U_\alpha)$ and $f_1, \ldots, f_n$ of $\mathcal{E}(U_\beta)$ are mapped to two bases of $\mathcal{E}(U_\alpha \cap U_\beta)$ by the restriction morphism. Because $\mathcal{E}(U_\alpha \cap U_\beta)$ is a free module, there exists a unique $\mathcal{C}^\infty(U_\alpha \cap U_\beta)$-valued matrix $g_{\beta \alpha}$, or, equivalently, a smooth map $g_{\alpha \beta} \colon U_\alpha \cap U_\beta \to \mathrm{GL}_n(\mathbb{R})$, such that $f_i = g_{\beta \alpha} e_i$. Almost by definition, these matrices satisfy the gluing conditions in \cref{subsec:construction-gluing-data}. The bundle $E$ is obtained by gluing the local bundles $\pi_\alpha \colon E_\alpha \to U_\alpha$ by the gluing data $g_{\beta \alpha}$.

There is an analogous theorem due to Swan (see \cite[Theorem 2]{swan_vector_1962}), which rephrases the correspondence of \Cref{thm:serre-bundle} in terms of projective modules.

\begin{example}
    We show an example of a subsheaf of vector fields that is not free. Consider the foliation $\mathcal{F}$ of $\mathfrak{so}^*(3)$ by concentric spheres or, equivalently, by the coadjoint action of $\mathrm{SO}(3)$. Let $\mathcal{E}$ be the sheaf of vector fields tangent to $\mathcal{F}$. Then, $\mathcal{E}$ is not locally free. Indeed, for a contractible open neighbourhood $U \subset M$ that does not contain the origin, $\mathcal{E}(U) \cong (\mathcal{C}^\infty(U))^2$. However, we know that if $V \subseteq$ is any open neighbourhood containing the origin, then $\mathcal{E}(V) \neq (\mathcal{C}^\infty(V))^2$. Indeed, any such isomorphism would yield two nowhere vanishing vector fields tangent to $\mathbb{S}^2$, which cannot be possible. 
\end{example}
  
\begin{example} \label{ex:b-manifold}
    Let $M$ be an orientable smooth manifold and consider an embedded oriented hypersurface $i \colon Z \to M$. Identify from now on $Z$ with its image $i(Z)$. Moreover, assume that $Z$ is defined as $Z = f^{-1}(0)$ for a locally defined function $f \colon U \to \mathbb{R}$ transverse to $Z$. We consider the subsheaf $\bi{\mathfrak{X}}(M) \subseteq \mathfrak{X}(M)$ of vector fields tangent to $Z$. If $p \notin Z$, we can find an open chart $(U, \varphi)$ centered at $p$ such that $U \cap Z = \emptyset$. Consequently, the condition of tangency is vacuous and $\bi{\mathfrak{X}}(U) = \mathfrak{X}(U)$. On the other hand, if $p \in Z$ we can take an adapted chart $(U, \varphi)$ with coordinates $x_1, \ldots, x_n$ such that $\varphi(Z \cap U) = \{x_1 = 0\}$. Therefore, the condition of tangency implies
    \begin{equation} \label{eq:loc-gen-b}
        \bi{\mathfrak{X}}(U) = \operatorname{span}_{\mathcal{C}^\infty(U)} \Big\{ x_1 \frac{\partial}{\partial x_1}, \frac{\partial}{\partial x_2}, \ldots, \frac{\partial}{\partial x_n} \Big\}.
    \end{equation}

    As a consequence of the previous discussion, the sheaf $\bi{\mathfrak{X}}$ is locally free and finitely generated. By \Cref{thm:serre-bundle} there exists a vector bundle $\bi{\mathrm{T}}M$, called the \emph{$b$-tangent bundle}, such that $\bi{\mathfrak{X}} = \Gamma \bi{\mathrm{T}} M$. This object was originally considered by Melrose \cite{melrose_atiyah-patodi-singer_1993} in the smooth setting in order to generalize the Atiyah-Patodi-Singer index theorem to manifolds with boundary. In our setting, the embedded hypersurface $Z$ plays the role of boundary. These structures were studied in the Poisson setting by Guillemin, Miranda, and Pires \cite{guillemin_symplectic_2014}. In this context, the set $Z$ arises as the singularity set of a Poisson structure $\Pi$: hence, we will often refer to it as the \emph{critical set}. The orientability conditions imposed on $M$ and $Z$ at the beginning of the example arise from this consideration.

   Such structures were previously studied in the algebraic setting, where they were introduced by Deligne \cite{deligne_equations_1970} under the name of \emph{forms with logarithmic poles}. $b$-Symplectic structures are also referred to as \emph{$b$-Poisson}, \emph{log-symplectic}, or \emph{log-Poisson} structures; see, for instance, \cite{gualtieri_symplectic_2014}.
\end{example}

\begin{remark} \label{rmk:trivial-normal-bundle}
    We have already mentioned that the orientability conditions on $M$ and $Z$ arise from symplectic considerations. The fact that $Z$ is the zero-level set of a function $f$, locally defined near $Z$ and transverse to it, has a direct geometric interpretation: this condition holds if and only if the normal bundle $\mathrm{N}Z$ is trivial.
    
    The proof of this remark is a straightforward application of the tubular neighborhood theorem. Let $\varphi \colon U \to \mathrm{N}Z$ be a tubular neighborhood diffeomorphism. If $f \colon U \to \mathbb{R}$ is a local defining function,\footnote{We assume that the domains of $\varphi$ and $f$ coincide. Otherwise, we may replace them by their intersection $\operatorname{dom}\varphi \cap \operatorname{dom}f$, which is an open neighborhood of $Z$.} then the composition
    \[
        f \circ \varphi^{-1} \colon \mathrm{N}Z \to \mathbb{R}
    \]
    defines a global section of $\mathrm{N}Z$ by transversality. Conversely, if $\mathrm{N}Z \cong \mathbb{R} \times Z$, the composition $\operatorname{pr}_1 \circ \varphi \colon U \to \mathbb{R}$ yields a transverse local defining function for $Z$. Finally, the normal bundle $\mathrm{N}Z$ may be identified with the kernel $\mathbb{L}$ of the map $\bi{\mathrm{T}}M \to \mathrm{T}M$ (see \cite[Proposition~4]{guillemin_symplectic_2014}).    
\end{remark}

\begin{example} \label{ex:bm-manifold}
    In the assumptions of \Cref{ex:b-manifold}, we can consider the subsheaf $\bi[m]{\mathfrak{X}} \subseteq \mathfrak{X}$ of vector fields which are tangent to $M$ with order at least $m$. Mathematically, we say $X \in \bi[m]{\mathfrak{X}}(M)$ if and only if $\mathcal{L}_X f \in \mathcal{I}_Z^m$. The set $\mathcal{I}_Z$ is the \emph{ideal sheaf} of $M$,
    \begin{equation*}
        \mathcal{I}_Z = \Set* { f \in \mathcal{C}^\infty(M) \given f(p) = 0 \text{ for all } p \in Z }.
    \end{equation*}
 
The defining equation \( \mathcal{L}_X f \in \mathcal{I}_Z^m \) depends on the choice of the defining function and, more precisely, on the \( (m - 1) \)-jet of \( f \) at \( Z \). This choice must be included as part of the defining data. The triple \( (M, Z, f) \) is referred to as a \emph{\( b^m \)-manifold}. Typically, the defining function is implicitly assumed and therefore not explicitly stated.

    As before, the sheaf $\bi[m]{\mathfrak{X}}$ is locally finitely generated and free, and in an adapted chart $(U, \varphi)$ for $Z$ with coordinates $x_1, \ldots, x_n$ we have the explicit generators
    \begin{equation} \label{eq:loc-gen-bm}
        \bi[m]{\mathfrak{X}}(U) = \operatorname{span}_{\mathcal{C}^\infty(U)} \Big\{ x_1^m \frac{\partial}{\partial x_1}, \frac{\partial}{\partial x_2}, \ldots, \frac{\partial}{\partial x_n} \Big\}.
    \end{equation}
    The Serre-Swan theorem implies now the existence of a vector bundle $\bi[m]{\mathrm{T}} M$ satisfying $\bi[m]{\mathfrak{X}}(M) = \Gamma \bi[m]{\mathrm{T}}M$. By analogy with \Cref{ex:bm-manifold}, this bundle is called the \emph{$b^m$-tangent bundle}.

    These objects were originally introduced by Scott \cite{scott_geometry_2016} as a generalization of $b$-manifolds. In analogy with $b$-manifolds, they have been used to define and study Poisson structures under the setting of singular symplectic geometry (cf. \cite{scott_geometry_2016, guillemin_desingularizing_2019, guillemin_geometric_2021}). Moreover, they naturally appear in problems of physical relevance, such as the study of escape orbits at infinity in the restricted, planar, circular three-body problem \cite{miranda_singular_2021} and other physical problems \cite{mirmirandanicolas}. We also remark that these objects were recently generalized by Bischoff, del Pino and Witte \cite{bischoff_jets_2023} by replacing a hypersurface by a submanifold.
\end{example}

\begin{example}\label{ex:2.6}
    Let us consider, as in the previous examples, a smooth manifold $M$ with an embedded hypersurface $Z$. We assume, additionally, that there is a fibration $\pi \colon Z \to N$. We call such data an \emph{edge structure}. In this case, the relevant subsheaf ${}^e{\mathfrak{X}}(M) \subseteq \mathfrak{X}(M)$ consists of vector fields, called \emph{edge vector fields}, which are not only tangent to the hypersurface $Z$ but also to the fibres of $\pi$. The sheaf of edge vector fields is locally free and finitely generated: in a local chart $(U, \varphi)$ with coordinates $x_1, \ldots, x_n$ adapted to $Z$ and such that the fibres of the fibration $\pi$ are described by $\{x_{n - k} = \ldots = x_n = c\}$, we have
    \begin{equation*}
        {}^e{\mathfrak{X}}(U) = \operatorname{span}_{\mathcal{C}^\infty(U)} \Big\{ x_1 \frac{\partial}{\partial x_1}, x_1 \frac{\partial}{\partial x_2}, \ldots, x_1 \frac{\partial}{\partial x_{n - k - 1}}, \frac{\partial}{\partial x_{n - k}}, \ldots, \frac{\partial}{\partial x_n} \Big\}.
    \end{equation*}
    The vector bundle obtained from \Cref{thm:serre-bundle} is called the \emph{edge tangent bundle}.

    Edge structures appear naturally in twistor theory (see section 2.5 of \cite{fine_knots_2022}). Consider the twistor space $(Z, J)$ of $\mathbb{H}^4$, where $J$ is the Eells--Salamon almost-complex structure. This almost-complex structure blows up when approaching the set $Z$. The edge structure associated to the twistor projection $Z \to \mathbb{S}^3$ gives an interpretation of $J$ as a \emph{smooth} almost-complex structure over the edge tangent bundle ${}^e{\mathrm{T}} M$. 
    
    Moreover, edge structures generalize $b$-manifolds (taking the fibration $\pi \colon Z \to p$ over a single point) and zero structures (taking the fibration $\pi \colon Z \to Z$ to be the identity map, see  \cite[Definition 2.13]{fine_knots_2022}).
\end{example}

All the previous examples fall under the general framework of \emph{$E$-manifolds}. These structures are instances of Lie algebroids $(\mathcal{A}, \rho, [\cdot,\cdot]_\mathcal{A})$ with generically injective anchor map $\rho$. In this sense, the geometry of an $E$-manifold is completely determined by the subsheaf $\rho(\mathcal{A}) \subseteq \mathrm{T} M$. These objects are also known in the literature as \emph{projective foliations} or \emph{Debord foliations} (due to their study by Debord \cite{debord_holonomy_2001}). The definition of $E$-manifold follows the work of Nest and Tsygan \cite{nest_deformations_2001}  in deformation quantization of symplectic Lie algebroids. These objects were abstractly defined and studied by Miranda and Scott \cite{miranda_geometry_2021}.

\subsection{Roadmap} \label{subsec:outline-proof}

The rest of the article will be devoted to studying the isomorphism classes of the bundles introduced in \Cref{ex:b-manifold,ex:bm-manifold}. More specifically, we discuss under which circumstances these bundles are isomorphic to the standard tangent bundle $\mathrm{T} M$. For the sake of simplicity, the base manifolds shall be assumed to be orientable. The criterion we use rests on a general conceptual framework and, as such,  we briefly pause to outline the structure of the discussion.

The main result we use in the discussion of the isomorphism classes of the bundles $\bi[m]{\mathrm{T}} M$ is the bundle construction theorem, briefly described in \cref{subsec:construction-gluing-data}. Thus, we need the expression of the bundles $\bi[m]\mathrm{T} M$ in terms of an open cover $\mathcal{U} = \{U_\alpha\}$ of $M$ and a collection of \emph{common} bundles $\pi_\alpha \colon E_\alpha \to U_\alpha$. \Cref{sec:isomorphism-bundles} is devoted to constructing such an open cover, as well as the computation of the gluing data $g_{\alpha \beta}$. This construction elucidates the definition of the $b$-tangent bundle: Although the Serre-Swan theorem gives a criterion for the existence of the bundles $\bi[m]{\mathrm{T}} M$, the explicit data are not known. As a bypass of these computations, we obtain the isomorphisms $\bi[2k]{\mathrm{T}}M \cong \mathrm{T}M$ and $\bi[2k + 1]{\mathrm{T}}M \cong \bi{\mathrm{T}} M$. We reinterpret this result in light of the desingularization procedure of Guillemin, Miranda, and Weitsman \cite{guillemin_desingularizing_2019}.

Once in the setting of the vector bundle construction theorem, the criterion for an isomorphism $\bi[m]{\mathrm{T}} M \cong \mathrm{T} M$ is the existence of the gauge transformations $f_\alpha \in \mathcal{G}(E_\alpha)$ satisfying the compatibility conditions described in \cref{subsec:construction-gluing-data}. We show that the local transformations $f_\alpha |_{U_\beta} \in \mathcal{G}(E_\alpha)|_{U_\beta}$ do always exist. Consequently, the existence of such $f_\alpha$ becomes an instance of an \emph{extension problem}. Discerning whether there exists a solution of the extension problem posed by the isomorphism $\bi{\mathrm{T}}M \cong \mathrm{T}M$ is very complicated in general.

In \cref{sec:colorability} we obtain combinatorial obstructions encoded in the associated graph $G_{M, Z}$ to the existence of an isomorphism between the $b$-tangent and tangent bundles. This result can be seen as the very first obstruction in the extension problem posed by the set $\pi_0(\operatorname{GL}_n(\mathbb{R})) = \mathbb{Z}/2\mathbb{Z}$. The criterion we obtain is good enough to solve the isomorphism problem in dimension 1, where the topological obstructions do not play a role.

Combinatorial obstructions are closely related to the characteristic classes of $\bi{\mathrm{T}}M$. Building on previous calculations of Cannas, Guillemin, and Woodward \cite{CGW} and Klaasse \cite{klaasse_geometric_2017}, we show the equality $\mathrm{w}(\bi{\mathrm{T}}M) = \mathrm{w}(\mathrm{T}M)$ can be read from the colorability of $G_{M, Z}$. Using similar computations for Pontrjagin classes, we deduce a criterion for the equality of all stable characteristic classes in terms of $G_{M, Z}$. We show a similar result for the equality $[\bi{\mathrm{T}}M] = [\mathrm{T}M]$ in the real $K$-theory ring $\mathrm{KO}(M)$.

The method used in \cref{sec:char-classes} allows the computation of stable characteristic classes. The Euler class, being the notorious example of a non-stable class, has to be computed using different methods. In \cref{sec:Euler-class} we present an adaptation of the classical Poincaré-Hopf theorem to compute the Euler number of a $b$-manifold in terms of the singularities of a $b$-vector field, although regarded as an honest vector field. This result provides a direct link between the topology of the pair $(M, Z)$ and the dynamics of $b$-vector fields.

In \Cref{sec:iso-b-spheres} we address the isomorphism problem from a global perspective in settings where a complete classification is possible. We first review low-dimensional cases, showing that in dimensions $1$, $3$, and $7$ the existence of an isomorphism $\mathrm{T}M \cong \bi{\mathrm{T}}M$ is entirely governed by the two-colorability of the associated graph $G_{M,Z}$, recovering and unifying known results.

We then turn to the model family of $b$-spheres $(\mathbb{S}^n,\mathbb{S}^{n-1})$, which serves as a test case for phenomena not detected by stable equivalence. Although all stable and cohomological obstructions vanish, we show that $\mathrm{T}\mathbb{S}^n$ and $\bi{\mathrm{T}}\mathbb{S}^n$ are isomorphic if and only if $n=1,3,$ or $7$.

The proof relies on an explicit computation of the clutching function of the $b$-tangent bundle, from which we deduce the stronger result that $\bi{\mathrm{T}}\mathbb{S}^n$ is always trivializable. This reveals a sharp distinction between stable and honest isomorphism and clarifies the role of classical parallelisability in the $b$-setting.

Finally, we apply these results to the existence of geometric structures on $b$-spheres, showing that the triviality of $\bi{\mathrm{T}}\mathbb{S}^n$ implies the existence of almost-symplectic, almost-contact, and, in even dimensions, almost-complex $b$-structures, despite the absence of corresponding structures on the standard tangent bundle. In section \ref{sec:edge-tangent}, we briefly discuss the isomorphism problem for specific instances of edge tangent bundles. We highlight a method to derive analogous obstructions, although, for general edge structures, discerning the existence of extensions becomes highly non-trivial. In section \ref{sec:concluding}, we conjecture the existence of analogous results for folded-cotangent bundles, and discuss further interplay between topology and dynamics of various singular systems.

We conclude this article with a technical appendix containing a proof of the triviality of the clutching function for the $b$-tangent bundle of spheres.

\section{Unbundling the singular bundles} \label{sec:isomorphism-bundles}

Following the outline developed in \cref{subsec:outline-proof}, we begin this section by constructing a common model for the tangent, $b$, and $b^m$-tangent bundles. Thus, we fix a given $b^m$-manifold $(M, Z,f)$ throughout the section with $f$ a defining function for $Z$. Our choice of open cover $\mathcal{U}$ of $M$ is strongly influenced by the semi-local structure of the critical set $Z$. We distinguish two different types of open sets.

Let us consider a tubular neighbourhood $\varphi \colon V \to \mathbb{R} \times Z$, and denote by $\varphi_\gamma \colon V_\gamma \to \mathbb{R} \times Z_\gamma$ the induced tubular neighbourhoods around the connected components $Z_\gamma$ of $Z$. We take the open and connected sets $V_\gamma$ to belong to the open cover $\mathcal{U}$. Regarding the model vector bundle $\pi_\gamma \colon E_\gamma \to V_\gamma$, we take
\begin{equation}
    E_\gamma = (V_\gamma \times \mathbb{R}) \oplus (\operatorname{pr}_2 \varphi_\gamma)^* \mathrm{T} Z_\gamma.
\end{equation}
For the sake of convenience, we will henceforth write $\mathcal{Z}_\gamma \coloneqq (\operatorname{pr}_2 \varphi_\gamma)^* \mathrm{T} Z_\gamma$.

The second type of open sets in $\mathcal{U}$ are modeled over the connected components of $M \setminus Z$. If $\tilde{U}_\alpha$ is such a connected component, we define the open set
\begin{equation}
    U_\alpha = \tilde{U}_\alpha \setminus \varphi^{-1}([- 1, 1] \times Z)
\end{equation}
Informally speaking, we are removing a ``distance-one neighbourhood'' to the critical set $Z$ of the connected component $\tilde{U}_\alpha$. As vector bundle $\pi_\alpha \colon E_\alpha \to U_\alpha$ we simply take $E_\alpha = \mathrm{T} U_\alpha$.

In the previous definitions, the collection of open and connected sets $\mathcal{U} = \{U_\alpha, V_\gamma\}$ is an open cover for $M$. The reader may find an example of the construction of such an open cover in \cref{fig:open-cover}.

\begin{figure}[t]
    \centering
    \includegraphics[width=0.47\linewidth]{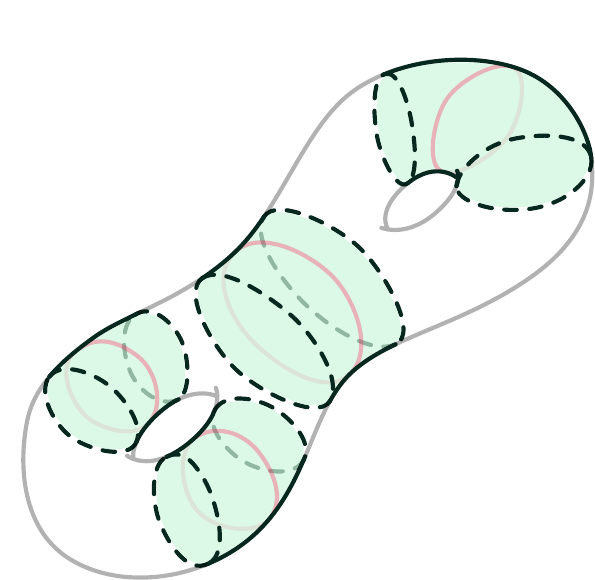}
    \includegraphics[width=0.47\linewidth]{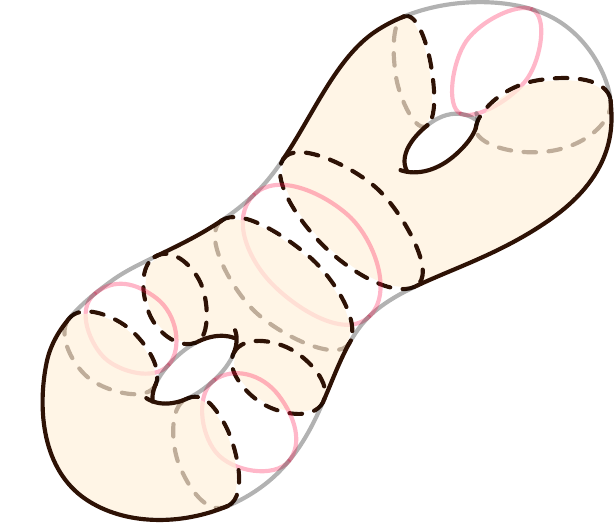}
    \caption{Images of a $b$-manifold $M$ (in black) with critical set $Z$ (in red) and an open cover in the assumptions of \cref{sec:isomorphism-bundles}. The collar neighbourhoods $V_\gamma$ are represented in green, while the open sets $U_\alpha$ are drawn in orange.}
    \label{fig:open-cover}
\end{figure}

\begin{remark}
    Observe that the triple intersection of any three different sets in $\mathcal{U}$ is automatically empty. The collection $\mathcal{U}$ is partitioned by the open sets $U_\alpha$ and $V_\gamma$. Therefore, any triple intersection contains at least two open sets $U_\alpha, U_\beta$ or $V_\gamma, V_\delta$. In the latter case, the intersection is either empty or $V_\gamma = V_\delta$. In the former case, if the intersection $U_\alpha \cap U_\beta$ is not empty, then $U_\alpha = U_\beta$ given both sets have to be contained in the same connected component of $M \setminus Z$.
    
    As a consequence, the gluing conditions for the gluing data that we obtain are fulfilled.
\end{remark}

\subsection{Gluing data for the tangent bundle} \label{subsec:gluing-tangent}

Let us consider the bundle $\mathrm{T}M$. Over a set $U_\alpha \in \mathcal{U}$, the tangent bundle $\mathrm{T} U_\alpha$ is simply the restriction or pullback bundle $\mathrm{T} U_\alpha = i_\alpha^* \mathrm{T} M$. Over a tubular neighbourhood $V_\gamma$, we can use the isomorphism $\varphi_\gamma$ to construct an isomorphism $\psi_\gamma \colon (V_\gamma \times \mathbb{R}) \times \mathcal{Z}_\gamma \to \mathrm{T} V_\gamma$ by taking $\psi_\gamma' = \diff \varphi^{-1}_\gamma$. In terms of elements, this map is described as
\begin{equation*}
    \psi_\gamma(g, Y) = g N_\gamma + \diff \varphi_\gamma^{-1}(Y),
\end{equation*}
where $N \in \mathfrak{X}(V_\gamma)$ is a normal vector field to $Z$ defined by $N = \diff \varphi_\gamma^{-1}(\partial / \partial t)$.

In order to better describe the trivialization data in terms of a matrix, our goal is to describe the restrictions $(i_{\alpha \gamma}^\alpha)^* \mathrm{T} U_\alpha$ and $(i_{\alpha \gamma}^\gamma)^* \mathrm{T} V_\gamma$ in terms of a common trivialization. For the latter case, the trivialization is induced by restriction of $\psi'_\gamma$ to the open set $U_\alpha \cap V_\gamma$, that is
\begin{equation*}
    \tau_\gamma' = (i_{\alpha \gamma}^\gamma)^* \psi_\gamma' \colon ((U_\alpha \cap V_\gamma) \times \mathbb{R}) \oplus (i_{\alpha \gamma}^\gamma)^* \mathcal{Z}_\gamma \longrightarrow (i_{\alpha \gamma}^\gamma)^* \mathrm{T} V_\gamma.
\end{equation*}
It turns out that the same model can be used as a trivialization of $\mathrm{T} U_\alpha$. To note so, we begin by realizing that $i_\alpha i_{\alpha \gamma}^\alpha = i_\gamma i_{\alpha \gamma}^\gamma$. Because $\mathrm{T} U_\alpha = i_\alpha^* \mathrm{T}M$ and $\mathrm{T} V_\gamma = i_\gamma^* \mathrm{T}M$ by definition, the universal property of the pullback bundle yields a unique isomorphism $\iota_{\alpha \gamma}' \colon (i_{\alpha \gamma}^\gamma)^* \mathrm{T} V_\gamma \to (i_{\alpha \gamma}^\alpha)^* \mathrm{T} U_\alpha$, which is nothing but the restriction of the identity map. Thus, the map $\iota_{\alpha \gamma}$ gives the gluing data of the local bundles and, at the level of sections, $\iota_{\alpha \gamma}'(X) = X$. We finally get a trivialization of $(i_{\alpha \gamma}^\alpha)^* \mathrm{T} U_\alpha$ as
\begin{equation*}
    \tau_\alpha' \coloneqq \iota_{\alpha \gamma}' \tau_\gamma' \colon ((U_\alpha \cap V_\gamma) \times \mathbb{R}) \oplus (i_{\alpha \gamma}^\gamma)^* \mathcal{Z}_\gamma \longrightarrow (i_{\alpha \gamma}^\alpha)^* \mathrm{T} U_\alpha.
\end{equation*}

Given the map $\iota_{\alpha \gamma}'$ corresponds to the gluing data of $\mathrm{T} M$, the corresponding gluing data $g_{\alpha \gamma}'$ in the trivialized bundles is given by the following commutative diagram:
\begin{equation*}
    \begin{tikzcd}[row sep=2.25em]
    	{((U_{\alpha} \cap V_\gamma) \times \mathbb{R}) \oplus (i_{\alpha \gamma}^\gamma)^* \mathcal{Z}_\gamma} & {(i_{\alpha \gamma}^\alpha)^* \mathrm{T} V_\gamma} \\
    	{((U_{\alpha} \cap V_\gamma) \times \mathbb{R}) \oplus (i_{\alpha \gamma}^\gamma)^* \mathcal{Z}_\gamma} & {(i_{\alpha \gamma}^\gamma)^* \mathrm{T} U_\alpha}
    	\arrow["{\tau_\gamma'}", from=1-1, to=1-2]
    	\arrow["{g_{\alpha \gamma}'}"', from=1-1, to=2-1]
    	\arrow["{\iota_{\alpha \gamma}'}", from=1-2, to=2-2]
    	\arrow["{\tau_{\alpha}'}"', from=2-1, to=2-2]
    \end{tikzcd}
\end{equation*}
it is straightforward, almost from the definition of $\tau_\alpha'$, that the gluing data is $g_{\alpha \gamma}' = \operatorname{id}$. We will, however, compute the matrix in terms of local sections as a warm-up for the upcoming subsection:
\begin{align*}
    g_{\alpha \gamma}'(g, Y) &= (\tau_{\alpha}')^{-1} \circ \iota_{\alpha \gamma}' \circ \tau_\gamma' (g, Y) \\
    &= (\tau_{\alpha}')^{-1} \circ \iota_{\alpha \gamma}' \big( g N_\gamma + \diff \varphi^{-1}_\gamma(Y) \big) \\
    &= (\tau_{\alpha}')^{-1} \big( g N_\gamma + \diff \varphi_\gamma^{-1}(Y) \big) \\
    &= (g, Y).
\end{align*}

\subsection{Gluing data for the \texorpdfstring{$b^m$}{bm}-tangent bundle} \label{subsec:gluing-bm-tangent}

As in the previous case, we consider the open cover $\mathcal{U}$. Over $U_\alpha \in \mathcal{U}$, we have $\bi[m]{\mathrm{T}} U_\alpha = \mathrm{T}U_\alpha$. Observe this equivalence is not an isomorphism, but rather an exact equality between bundles. Over a set $V_\gamma$, the tubular neighbourhood $\varphi_\gamma$ induces an isomorphism $\psi_\gamma \colon (V_\beta \times \mathbb{R}) \oplus \mathcal{Z}_\gamma \to \bi[m]{\mathrm{T}} V_\gamma$ in a similar spirit to $\psi_\gamma'$. In terms of elements, this map is
\begin{equation*}
    \psi_{\gamma} (g, Y) = g f^m N_\gamma + \diff \varphi_\gamma^{-1}(Y).
\end{equation*}
Here, $f$ is any local defining function for $Z_\gamma$ (with appropriate restrictions in the $(m-1)$-jet). We make the following choice. Let us consider a smooth function $h \colon \mathbb{R} \to \mathbb{R}$ satisfying $h = 1$ at $[1/2, +\infty)$, $h = -1$ at $(- \infty, -1/2]$, $h'(0) \neq 0$ and $h^{-1}(0) = \{0\}$ (see \cref{fig:def-funct} for an example of such a function). Then, we define $f \coloneqq (\operatorname{pr}_1 \varphi_\gamma)^* h$. As a consequence, $f$ is a local defining function for $Z_\gamma$ and the map $\psi_\gamma$ is an isomorphism of bundles.

\begin{figure}[t]
    \centering
    \includegraphics[width=0.65\linewidth]{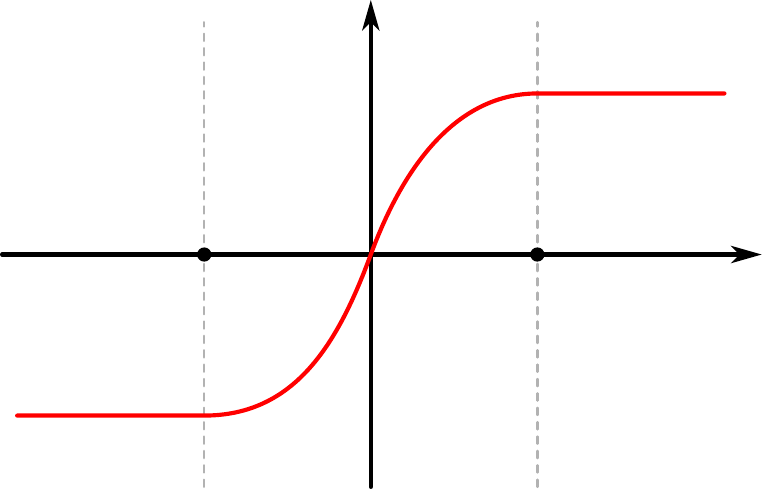}
    \caption{Depiction of the local defining function for $Z$ used in \cref{subsec:gluing-bm-tangent}. This choice can be further adapted to additional requirements: we may assume $f \neq 1, -1$ only inside the open set $(- \varepsilon, \varepsilon)$. Moreover, this choice can be made so that $f'(0) = 1$ always.}
    \label{fig:def-funct}
\end{figure}

In terms of trivializations of these bundles, the fact that $\bi[m] \mathrm{T} U_\alpha = \mathrm{T} U_\alpha$ yields an isomorphism
\begin{equation*}
    \tau_\alpha = \tau_\alpha' \colon ((U_\alpha \cap V_\gamma) \times \mathbb{R}) \oplus (i_{\alpha \gamma}^\gamma)^* \mathcal{Z}_\gamma \longrightarrow (i_{\alpha \gamma}^\alpha)^* \bi[m]{\mathrm{T}} U_\alpha
\end{equation*}
Similarly to the previous construction, the restriction of $\psi_\gamma$ induces an isomorphism
\begin{equation*}
    \tau_\gamma = (i_{\alpha \gamma}^\gamma)^* \psi_\gamma \colon ((U_\alpha \cap V_\gamma) \times \mathbb{R}) \oplus \mathcal{Z}_\gamma \longrightarrow (i_{\alpha \gamma}^\gamma)^* \bi[m]{\mathrm{T}} V_\gamma.
\end{equation*}

The gluing data $g_{\alpha \gamma}$ is obtained from the restriction $\iota_{\alpha \gamma} \colon (i_{\alpha \gamma}^\gamma)^* \bi[m]{\mathrm{T}} V_\gamma \to (i_{\alpha \gamma}^\alpha)^* \bi[m]{\mathrm{T}} U_\alpha$ of the identity map. Consequently, in the given trivializations this map reads
\begin{align*}
    g_{\alpha \gamma}(g, Y) &= \tau_\alpha^{-1} \circ \iota_{\alpha \gamma} \circ \tau_\gamma (g, Y) \\
    &= \tau_\alpha^{-1} \circ \iota_{\alpha \gamma} \big( g f^m N_\gamma + \diff \varphi^{-1}_\gamma (Y) \big) \\
    &= \tau_\alpha^{-1} \big( g f^m N_\gamma + \diff \varphi^{-1}_\gamma (Y) \big) \\
    &= \big( (\pm 1)^m g, Y \big).
\end{align*}
Here, we have used the fact that $f|_{U_\alpha \cap V_\gamma} = \pm 1$.

The construction of the $b$-tangent bundle using gluing data, as mentioned in \cite{cannas_da_silva_fold-forms_2010}, provides a valuable foundation. Building on this, there is significant potential to further explore and understand the topologies of these bundles.

\subsection{Isomorphisms of \texorpdfstring{$b^m$}{bm}-tangent bundle and desingularization of \texorpdfstring{$b^m$}{bm}-symplectic structures}

We conclude this section by stating and proving the main isomorphism theorem for the various bundles discussed.

\begin{theorem} \label{thm:isomorphism-bm-bundles}
    Let $(M, Z, f)$ be a $b^m$-manifold with $f$ a defining function for $Z$.
    \begin{itemize}
        \item  If $m = 2k$, then the $b^m$-tangent bundle is isomorphic to the tangent bundle for any choice of $f$, that is, $\bi[m]{\mathrm{T}} M \cong \mathrm{T} M$.
        \item  If $m = 2k + 1$, then the $b^m$-tangent bundle is isomorphic to the $b$-tangent bundle, that is, $\bi[m]{\mathrm{T}} M \cong \bi{\mathrm{T}} M$.
    \end{itemize}
\end{theorem}

\begin{proof}
    The proof is a direct consequence of the gluing conditions found in \cref{subsec:gluing-tangent,subsec:gluing-bm-tangent}. The explicit isomorphism $g \colon \mathrm{T} M \to \bi[m]{\mathrm{T}} M$ when $m = 2k$ is obtained by gluing the respective identifications: in terms of local generators around $V_\gamma$, we have $g( h N_\gamma + \diff \varphi^{-1}_\gamma (Y) ) = f^{2k} h N_\gamma + \diff \varphi_\gamma^{-1}(Y)$. Similar formulae apply when $m = 2k + 1$.
\end{proof}

The previous isomorphisms are closely related to the desingularization procedure of Guillemin, Miranda, and Weitsman \cite{guillemin_desingularizing_2019} for $b^m$-symplectic forms. Roughly speaking, a $b^m$-symplectic form is a symplectic form $\omega \in \Omega^2(M \setminus Z)$ admitting a semi-local decomposition as
\begin{equation}
    \omega = \frac{\diff f}{f^m} \wedge \alpha + \beta,
\end{equation}
for some forms $\alpha \in \Omega^1(M)$ and $\beta \in \Omega^2(M)$, which can further be assumed to be closed. Here, $f$ is the chosen semi-local defining function. Thus, the form $\omega$ can be interpreted as a singular symplectic form with singularities controlled by the defining function $f$.

In \cite{guillemin_desingularizing_2019}, it is shown that a $b^m$-symplectic form $\omega$ gives rise to a symplectic form when $m$ is even and to a $b$-symplectic form when $m$ is odd. Moreover, the proof in \cite{guillemin_desingularizing_2019} is constructive: forms $\omega_\varepsilon$ are explicitly constructed to achieve the desired desingularizations. In this work, we demonstrate how \Cref{thm:isomorphism-bm-bundles} provides an alternative construction of these forms.

\begin{corollary}
    Let $(M, Z,f)$ be a $b^m$-symplectic manifold and $\omega$ a $b^m$-symplectic form. If $m$ is even and $g \colon \mathrm{T} M \to \bi[m]{\mathrm{T}} M$ is the isomorphism of \Cref{thm:isomorphism-bm-bundles}, then $g^*(\omega)$ is a symplectic form.
\end{corollary}

\begin{proof}
    By virtue of the Moser path method in $b^m$-symplectic geometry (confer \cite[Theorem 5.2]{scott_geometry_2016}), we may choose a tubular neighbourhood $\varphi$ of $Z$ in which $\omega$ is
    \begin{equation*}
        \omega = \frac{\diff f}{f^{2k}} \wedge (\operatorname{pr}_2 \varphi_\gamma)^* \alpha + (\operatorname{pr}_2 \varphi_\gamma)^* \beta.
    \end{equation*}
    Therefore, the expression of the resulting two-form is
    \begin{equation*}
        g^*(\omega) = \diff f \wedge (\operatorname{pr}_2 \varphi_\gamma)^* \alpha + (\operatorname{pr}_2 \varphi_\gamma)^* \beta.
    \end{equation*}
    This form is closed and non-degenerate, hence symplectic. In an open set $U_\alpha$, the forms $g^*(\omega)$ and $\omega$ agree.
\end{proof}

The same construction applies step by step to the case when $m = 2k + 1$.

\section{Colorability and isomorphism of \texorpdfstring{$b$}{b}-type bundles} \label{sec:colorability}

\begin{figure}[t]
    \centering
    \includegraphics[width=.8\linewidth]{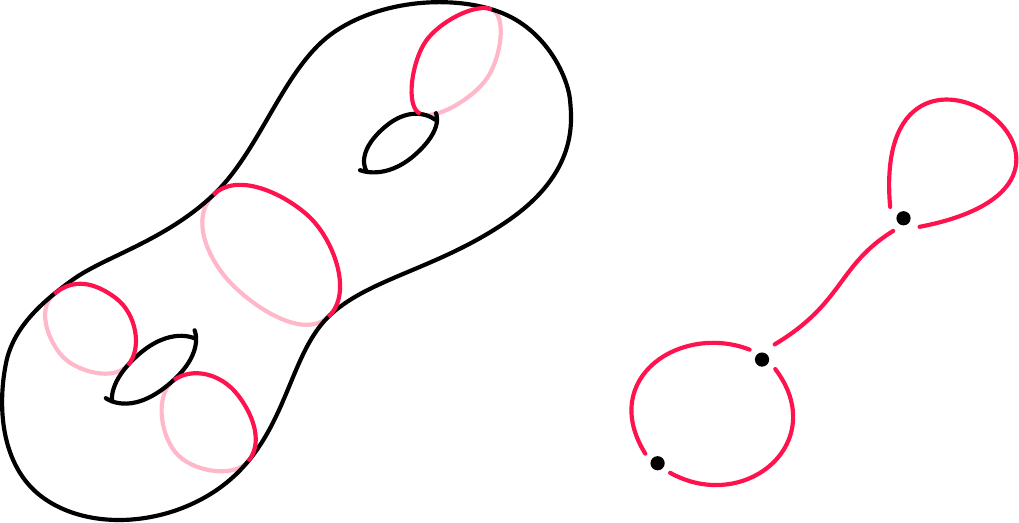}
    \caption{Example of a $b$-manifold $M$ with critical set $Z$ (in red). On the right, we have its associated graph. Observe one edge can go from one node to itself: in particular, this graph is not two-colorable.}
    \label{fig:ass-graph}
\end{figure}

As a consequence of \Cref{thm:isomorphism-bm-bundles}, the problem of finding the isomorphism classes of $b^m$-bundles reduces to two distinct cases: the isomorphism class of $\mathrm{T}M$ and that of $\bi{\mathrm{T}}M$. Consequently, we will focus solely on determining whether these two classes agree or differ.

In this section, we discuss combinatorial obstructions to the existence of an isomorphism $\mathrm{T}M \cong \bi{\mathrm{T}}M$. To this end, we introduce the associated graph of a $b$-manifold $(M, Z)$. Our construction of the open cover $\mathcal{U}$ in \cref{sec:isomorphism-bundles} is inspired by this object. These graphs were originally considered in \cite{miranda_equivariant_2017} and have since found applications in \cite{kirchhoff-lukat_log_2023} and \cite{brugues_arnold_2024}.

\begin{definition} \label{def:assoc-graph}
    Let $(M, Z)$ be a $b$-manifold. The \emph{associated graph} of $(M, Z)$ is the graph $G_{M, Z} = (N, E)$ with set of nodes $N = \{M_\alpha\}$, the set of connected components of $M \setminus Z$, and set of edges $E = \{ Z_\gamma \}$, the set of connected components of $Z$. In this definition, we say $M_\alpha \in Z_\gamma$ if and only if $\overline{M}_\alpha \cap Z_\gamma \neq \emptyset$. 
\end{definition}

\begin{remark}
    The associated graph $G_{M, Z}$ is, indeed, a graph. To see this we have to check that any edge $Z_\gamma$ contains only two nodes $M_\alpha, M_\beta$. This is a consequence of \Cref{rmk:trivial-normal-bundle}. Indeed, if we take a tubular neighbourhood $\varphi \colon U \to Z \times \mathbb{R}$, the connected and open sets $\varphi^{-1}(Z \times \mathbb{R}_+)$ and $\varphi^{-1}(Z \times \mathbb{R}_-)$ are contained in two unique connected components $M_\alpha$, $M_\beta$ of $M \setminus Z$ (observe these do not have to be distinct, cf. \cref{fig:ass-graph}).
\end{remark}

Our construction of the open cover $\mathcal{U}$ in \cref{sec:isomorphism-bundles} is inspired by the associated graph $G_{M, Z}$ in the following way. The open sets $U_\alpha$ represent the nodes of $G_{M, Z}$, while the open sets $V_\gamma$ are in correspondence with the edges. If we denote by $[U_\alpha] \in N$ and $[V_\gamma] \in E$ the corresponding nodes and edges, the condition $[U_\alpha] \in [V_\gamma]$ is equivalent to $U_\alpha \cap V_\beta \neq \emptyset$. Thus, adjacency data in the graph $G_{M, Z}$ can be equivalently read from the open cover $\mathcal{U}$.

With the previous definition, we are now in a position to provide a first obstruction to the isomorphism between tangent and $b$-tangent bundles. We then show how this simple criterion recovers some known concrete results.

\begin{proposition} \label{prop:comb-obstr}
    Let $(M, Z)$ be a $b$-manifold. Whenever there is an isomorphism $\bi{\mathrm{T}}M \cong \mathrm{T}M$, the associated graph $G_{M, Z}$ is two-colorable.\footnote{
        We say a graph $G = (N, E)$ is \emph{two-colorable} if there exists a function $c \colon N \to \{+, -\}$ such that, whenever $a, b \in e$ for some edge $e \in E$, we have $c(a) \neq c(b)$.
    }
\end{proposition}

\begin{remark}
    The converse implication does not hold. We can see this in a very specific example: consider the manifold $M = \mathbb{S}^2$ with the hypersurface $Z = \mathbb{S}^1 = \mathbb{S}^2 \cap \{z = 0\}$. The graph $G_{M, Z}$ consists of two nodes joined by a single edge, which is clearly two-colorable. However, in \cite[Example~2.4.47]{Brugues2024}, it is shown that the $b$-tangent bundle $\bi{\mathrm{T}} \mathbb{S}^2$ is parallelizable, which in particular implies there exists no isomorphism $\mathrm{T}\mathbb{S}^2 \cong \bi{\mathrm{T}}\mathbb{S}^2$.
\end{remark}

\begin{proof}
    We apply the vector bundle construction theorem discussed in \cref{subsec:construction-gluing-data}. The existence of an isomorphism $f \colon \mathrm{T} M \to \bi{\mathrm{T}} M$ implies the existence of gauge transformations $f_\alpha \in \mathcal{G}(E_\alpha)$. As the determinant function of a gauge transformation is well-defined, we have a map $\det f_\alpha \colon U_\alpha \to \mathbb{R}^*$. Moreover, because $U_\alpha$ is connected, we have $\operatorname{sign} \det f_\alpha$ is constant on $U_\alpha$. Thus, we define a coloring $c \colon N \to \{+, -\}$ by setting $c([U_\alpha]) = \operatorname{sign} \det f_\alpha$.

    To conclude the proof, we must show that this choice is a two-coloring. Let us choose the only two nodes $[U_\alpha], [U_\beta] \in [V_\gamma]$. Up to a permutation (or equivalently, a choice of sign of the defining function), the gluing data found in \cref{subsec:gluing-tangent,subsec:gluing-bm-tangent} is given by
    \begin{equation*}
        g_{\alpha \gamma} = \begin{pmatrix}
            1 & \\ & \operatorname{id}_{\mathrm{T}Z_\gamma}
        \end{pmatrix}, \quad g_{\beta \gamma} = \begin{pmatrix}
            -1 & \\ & \operatorname{id}_{\mathrm{T}Z_\gamma}
        \end{pmatrix}.
    \end{equation*}
    Analogously, we know the gluing data for the tangent bundle $\mathrm{T}M$ is given by $g_{\alpha \gamma}' = g_{\beta \gamma}' = \operatorname{id}$. The isomorphism criterion in \cref{subsec:construction-gluing-data} yields now $g_{\alpha \gamma} f_\gamma = f_\alpha g_{\alpha \gamma}'$ and $g_{\beta \gamma} f_{\gamma} = f_{\beta} g_{\beta \gamma}'$. Taking the determinant of both expressions, we obtain
    \begin{align*}
        \operatorname{sign} \det g_{\alpha \gamma} \cdot \operatorname{sign} \det f_\gamma &= \operatorname{sign} \det f_\alpha \cdot \operatorname{sign} \det g_{\alpha \gamma}', \\
        \operatorname{sign} \det g_{\beta \gamma} \cdot \operatorname{sign} \det f_\gamma &= \operatorname{sign} \det f_\beta \cdot \operatorname{sign} \det g_{\beta \gamma}'.
    \end{align*}
    The previous expressions readily imply $\operatorname{sign} \det f_\alpha = \operatorname{sign} \det f_\gamma$ and $\operatorname{sign} \det f_\beta = - \operatorname{sign} \det f_\gamma$. As a consequence, $c([U_\alpha]) = - c([U_\beta])$ and thus $c$ is a coloring of $G_{M, Z}$.
\end{proof}

In particular, we obtain,

\begin{proposition} \label{prop:btang-s1}
    Let $M = \mathbb{S}^1$ and let $Z$ be a discrete set of points. Then, $\bi{\mathrm{T}} \mathbb{S}^1 \cong \mathrm{T} \mathbb{S}^1$ if and only if $\# Z$ is even.
\end{proposition}

\begin{proof}
    \Cref{prop:comb-obstr} already shows the isomorphism cannot exist whenever $\# Z$ is odd: in this case, the associated graph $G_{M, Z}$ is not two-colorable.

    To prove the existence of an isomorphism $\bi{\mathrm{T}}M \cong \mathrm{T} M$ when $\# Z$ is even, we recall Lemma 4.3 in \cite{brugues_arnold_2024} implies the existence of a global defining function $f \in \mathcal{C}^\infty(\mathbb{S}^1)$ for $Z$. If $X \in \mathfrak{X}^1(\mathbb{S}^1)$ is a nowhere vanishing section of $\mathrm{T} \mathbb{S}^1$, then $fX$ defines a nowhere vanishing section of $\bi{\mathrm{T}}\mathbb{S}^1$.
\end{proof}

Partial results in the direction of Proposition \ref{prop:btang-s1} can be found in \cite[Section 2.4.3]{Brugues2024}.

\section{Characteristic classes} \label{sec:char-classes}

We begin the discussion by examining the invariants of the $b$ and $b^m$-tangent bundles. We will focus our attention on two specific types of invariants: the representatives in the K-theory ring and the characteristic classes. Note that, from \Cref{thm:isomorphism-bm-bundles}, we can restrict our attention to the bundles $\mathrm{T} M$ and $\bi{\mathrm{T}} M$. The study of such invariants is important because they provide obstructions to the existence of isomorphisms.

\subsection{Characteristic classes}

We begin the discussion with characteristic classes. Recall that all stable characteristic classes of a real vector bundle are determined by their Stiefel--Whitney and Pontrjagin classes. The relationship between the Stiefel--Whitney classes of the bundles $\bi{\mathrm{T}} M$ and $\mathrm{T} M$ was originally noticed by Cannas, Guillemin and Woodward in \cite{CGW}. Their result is based on the following proposition, which they mention in a footnote \cite{CGW} (see also \cite{cannas_da_silva_fold-forms_2010} when the hypersurface $Z$ is dividing). For a complete proof, the reader may also consult \cite[Proposition 11.1.1]{klaasse_geometric_2017}.

\begin{proposition}[Cannas, Guillemin and Woodward] \label{prop:klaasse-line-bundle}
    Let $(M, Z)$ be a $b$-manifold. Then, there exists a line bundle $L$ such that $\mathrm{T}M \oplus \mathbb{R} \cong \bi{\mathrm{T}} M \oplus L$.
\end{proposition}

\begin{remark} \label{rmk:construction-line-bundle}
    The bundle $L$ is a line bundle, possibly non-trivial, which we can explicitly construct using the gluing data in \cref{sec:isomorphism-bundles}. Over a tubular neighbourhood $V_\gamma$ of $Z_\gamma$ with nodes $U_\alpha$ and $U_\beta$, the gluing data for $L$ is given by $g_{\alpha \gamma} = 1$ and $g_{\beta \gamma} = - 1$.
\end{remark}

The usual relation $\mathrm{w}(E \oplus E') = \mathrm{w}(E) \smile \mathrm{w}(E')$ for the total Stiefel--Whitney class of vector bundles and the previous proposition imply that $\mathrm{w}(\mathrm{T}M) = \mathrm{w}(\bi{\mathrm{T}}M) \smile \mathrm{w}(L)$. Furthermore, $\mathrm{w}(L) = 1 + \mathrm{w}_1(L)$ and $\mathrm{w}_1(L) \smile \mathrm{w}_1(L) = 0$ because $L \oplus L$ is trivial, so that we have 
\begin{equation} \label{eq:sw-class-btang}
    \mathrm{w}(\bi{\mathrm{T}}M) = \mathrm{w}(\mathrm{T}M) \smile (1 + \mathrm{w}_1(L)).
\end{equation}
Cannas, Guillemin, and Woodward identify the class $w_1(L)$ as the reduction mod-2 of the Poincaré dual of the fundamental class of $Z$.

Formula \eqref{eq:sw-class-btang} was presented by Cannas, Guillemin, and Woodward \cite{CGW}, and also mentioned by Klaasse \cite{klaasse_geometric_2017}. The following proposition provides combinatorial insight into these previous computations.

\begin{proposition} \label{prop:equiv-linebund-swclass}
    Let $(M, Z)$ be a $b$-manifold. The following are equivalent:
    \begin{enumerate}
        \item\label{it:graph-two-colorable} The graph $G_{M, Z}$ is two-colorable.
        \item\label{it:global-def-function} The defining function $f$ is global.
        \item\label{it:triv-bundle} The bundle $L$ is trivial.
        \item\label{it:equ-stief-whit-class} The total Stiefel--Whitney classes of $\mathrm{T}M$ and $\bi{\mathrm{T}}M$ agree.
        \item\label{it:equ-first-stief-whit-class} The first Stiefel--Whitney classes of $\mathrm{T}M$ and $\bi{\mathrm{T}}M$ agree.
        \item\label{it:triv-det-bun} The vector bundle $\bi{\mathrm{T}}M$ is orientable.
    \end{enumerate}
\end{proposition}
\begin{proof}
The equivalence between \ref{it:graph-two-colorable} and \ref{it:triv-bundle} follows from the same argument as in the proof of \Cref{prop:comb-obstr}, together with the bundle construction theorem of \Cref{subsec:construction-gluing-data}.  
In the gluing data of \Cref{rmk:construction-line-bundle}, the bundle $L$ is trivial if and only if, for every edge $V_\gamma$ with adjacent nodes $U_\alpha$ and $U_\beta$, there exist nowhere-vanishing functions $f_\alpha,f_\beta,f_\gamma$ such that
\[
f_\alpha=f_\gamma, \qquad f_\beta=-f_\gamma .
\]
Such a choice determines a coloring of $G_{M,Z}$ by setting $c(U_\alpha)=\operatorname{sign}(f_\alpha)$. Conversely, any two-coloring $c$ determines such functions by setting $f_\alpha=c(U_\alpha)\cdot 1$.

The equivalence between \ref{it:global-def-function} and \ref{it:triv-bundle} was already observed by Brugués, Miranda, and Oms \cite[Lemma~4.5]{brugues_arnold_2024}.

We now show that \ref{it:equ-stief-whit-class} and \ref{it:equ-first-stief-whit-class} are equivalent.  
Clearly, \ref{it:equ-stief-whit-class} implies \ref{it:equ-first-stief-whit-class}. Conversely, by equation~\eqref{eq:sw-class-btang},
\[
\mathrm{w}_1(\bi{\mathrm{T}}M)
=
\mathrm{w}_1(\mathrm{T}M)+\mathrm{w}_1(L).
\]
Thus, the equality $\mathrm{w}_1(\mathrm{T}M)=\mathrm{w}_1(\bi{\mathrm{T}}M)$ holds if and only if $\mathrm{w}_1(L)=0$. Since real line bundles are classified by their first Stiefel--Whitney class, this is equivalent to $L$ being trivial. Using again \eqref{eq:sw-class-btang}, this implies equality of total Stiefel--Whitney classes.

Finally, we prove the equivalence between \ref{it:triv-bundle} and \ref{it:triv-det-bun}. The determinant bundles satisfy
\[
\det(\mathrm{T}M)\cong L\otimes \det(\bi{\mathrm{T}}M).
\]
Since $M$ is orientable, $\det(\mathrm{T}M)$ is trivial. Therefore, $L$ is trivial if and only if $\det(\bi{\mathrm{T}}M)$ is trivial, which is equivalent to $\bi{\mathrm{T}}M$ being orientable.
\end{proof}

\begin{remark}
    It may be natural to ask whether, under the assumption that $G_{M,Z}$ is two-colourable, the equality $\mathrm{T}M \oplus \mathbb{R} \cong \bi{\mathrm{T}}M \oplus \mathbb{R}$, which follows from item~\ref{it:triv-bundle} of \Cref{prop:equiv-linebund-swclass} together with \Cref{prop:klaasse-line-bundle}, implies that $\mathrm{T}M \cong \bi{\mathrm{T}}M$. This implication does not hold in general, since stable isomorphism is strictly weaker than isomorphism for vector bundles.
    
    Counterexamples to this phenomenon will appear throughout the remainder of the text. As a direct example, it is shown in \cite[Example~2.4.47]{Brugues2024} that the $b$-tangent bundle of $(\mathbb{S}^2, \mathbb{S}^1)$ is trivializable by exhibiting an explicit global frame. The associated graph $G_{\mathbb{S}^2,\mathbb{S}^1}$ consists of two vertices joined by a single edge and is therefore two-colourable. Nevertheless, an isomorphism $\mathrm{T}\mathbb{S}^2 \cong \bi{\mathrm{T}}\mathbb{S}^2 \cong \mathbb{S}^2 \times \mathbb{R}^2$ is obstructed by the Hairy Ball Theorem.
\end{remark}

The previous result shows that the obstructions to an isomorphism
\[
    \mathrm{T}M \cong \bi{\mathrm{T}}M
\]
encoded in Stiefel--Whitney classes are completely characterized by the combinatorics of the graph $G_{M,Z}$. Moreover, all these obstructions are intimately related to the existence of global defining functions for $Z$.

A first consequence of this proposition is the following:

\begin{corollary}\label{cor:b-orientability}
    Let $(M,Z)$ be a $b$-manifold. Then
    \[
        \bi{\mathrm{T}}M \text{ is orientable} \quad \Longleftrightarrow \quad G_{M,Z} \text{ is two-colorable}.
    \]
    In particular, if $M$ is orientable and $G_{M,Z}$ is not two-colorable, then $\bi{\mathrm{T}}M$ is \emph{not} orientable.
\end{corollary}

\begin{proof}
    This is precisely the equivalence between \ref{it:graph-two-colorable} and \ref{it:triv-det-bun} in \Cref{prop:equiv-linebund-swclass}.
\end{proof}

\begin{example}
    Consider the two-torus with a single embedded circle as \( Z \). In this situation, the associated graph is not two-colorable, and consequently the \( b \)-tangent bundle is non-orientable.
\end{example}

As a consequence of the previous characterization, we obtain criteria for the existence of geometric structures on $b$-manifolds $(M,Z)$ in terms of the colorability of the graph $G_{M,Z}$. We begin by introducing the notions of almost $b$-symplectic and almost $b$-contact structures.

\begin{definition}
    A $b$-two-form $\omega \in \bi{\Omega}^2(M)$ on a $b$-manifold $(M,Z)$ is called \emph{almost $b$-symplectic} if it is non-degenerate at every point, but not necessarily closed.
\end{definition}

Recall the definition of $b$-contact structure.

\begin{definition}
    Let $(M,Z)$ be a $(2n+1)$-dimensional $b$-manifold. A $b$-one-form $\alpha \in \bi{\Omega}^1(M)$ is called \emph{ $b$-contact} if
    \[
        \alpha \wedge (\diff \alpha)^n
    \]
    is nowhere vanishing as a section of $\Lambda^{2n+1}(\bi{\mathrm{T}}^*M)$.
\end{definition}

The definition of an almost $b$-contact structure is given below. The key distinction is that the form $\beta$ is not required to coincide with $\diff \alpha$.

\begin{definition}
    Let $(M,Z)$ be a $(2n+1)$-dimensional $b$-manifold. A $b$-one-form $\alpha \in \bi{\Omega}^1(M)$ is called \emph{almost $b$-contact} if there exists a $b$-form of degree 2 $\beta\in \bi{\Omega}^2(M)$ such that
    \[
        \alpha \wedge \beta^{n}
    \]
    is a nowhere-vanishing $(2n+1)$-$b$-form on $M$.
\end{definition}

\begin{remark}
    Equivalently, an almost $b$-contact structure on $(M,Z)$ can be characterized in stable terms. Namely, $(M,Z)$ admits an almost $b$-contact structure if and only if the stabilized $b$-tangent bundle
    \[
        \bi{\mathrm{T}}M \oplus {\mathbb{R}}
    \]
    admits an almost $b$-symplectic structure. Here ${\mathbb{R}}$ denotes the trivial real line bundle. In particular, all odd-degree Stiefel--Whitney classes of $\bi{\mathrm{T}}M$ vanish if and only if all odd-degree Stiefel--Whitney classes of $\bi{\mathrm{T}}M \oplus \mathbb{R}$ vanish, so that the topological obstructions to the existence of almost $b$-contact structures are detected already at the level of $\bi{\mathrm{T}}M$.
\end{remark}

\begin{proposition}
    Let $(M,Z)$ be a $b$-manifold and let $G_{M,Z}$ be its associated graph. If $(M,Z)$ admits an almost $b$-symplectic (resp.\ almost $b$-contact) structure, then the graph $G_{M,Z}$ is two-colorable.
\end{proposition}

\begin{proof}
    An almost $b$-symplectic (resp.\ almost $b$-contact) structure equips the $b$-tangent bundle $\bi{\mathrm{T}}M$ with an almost symplectic (resp.\ almost contact) structure. A standard topological necessary condition for the existence of either structure is the vanishing of the first Stiefel--Whitney class \cite{libermann-complex-1955}:
    \[
        \mathrm{w}_1(\bi{\mathrm{T}}M)=0.
    \]
    (For the present argument, we will only use this degree-one obstruction.)

    On the other hand, the Stiefel--Whitney classes of $\bi{\mathrm{T}}M$ and $TM$ are related by
    \[
        \mathrm{w}_1(\bi{\mathrm{T}}M)=\mathrm{w}_1(\mathrm{T}M)+\mathrm{w}_1(L),
    \]
    where $L$ is the real line bundle canonically associated with the $b$-structure. If the graph $G_{M,Z}$ is not two-colorable, then by \Cref{prop:equiv-linebund-swclass} the line bundle $L$ is non-trivial, and hence
    \[
        \mathrm{w}_1(L)\neq 0.
    \]
    It follows that $\mathrm{w}_1(\bi{\mathrm{T}}M)\neq 0$, contradicting the necessary condition above. Therefore $G_{M,Z}$ must be two-colorable.
\end{proof}

\begin{remark}
    Almost symplectic and almost contact structures also force the vanishing of all higher odd-degree Stiefel--Whitney classes. However, the obstruction detected here is already visible at the level of the first Stiefel--Whitney class and is completely encoded by the combinatorics of the graph $G_{M,Z}$.
\end{remark}

\begin{remark}
    The results of this section show that the existence of almost $b$-symplectic and almost $b$-contact structures on a $b$-manifold $(M,Z)$ is governed by a purely combinatorial invariant: the two-colorability of the associated graph $G_{M,Z}$.
    
    More precisely, the obstruction to identifying the tangent bundle $\mathrm{T}M$ with the $b$-tangent bundle $\bi{\mathrm{T}}M$ is entirely encoded in the real line bundle $L$ associated with the $b$-structure, whose triviality is equivalent to the vanishing of $\mathrm{w}_1(L)$. Through the relation
    \[
        \mathrm{w}(\bi{\mathrm{T}}M)=\mathrm{w}(\mathrm{T}M)\smile \mathrm{w}(L),
    \]
    this obstruction manifests itself as a discrepancy between the Stiefel--Whitney classes of $\mathrm{T}M$ and $\bi{\mathrm{T}}M$.

    Since almost symplectic and almost contact structures impose the vanishing of all odd Stiefel--Whitney classes, the failure of $G_{M,Z}$ to be two-colorable provides a sharp and computable obstruction to the existence of such structures in the $b$-category. In particular, this shows that the presence of a hypersurface $Z$ can introduce genuinely new topological obstructions that do not appear in the smooth setting.

    These results highlight a striking phenomenon: geometric structures on $b$-manifolds, which a priori depend on differential-geometric data, may be obstructed by discrete and combinatorial features of the singular locus.
\end{remark}

We now consider the relation between the Pontrjagin classes of the tangent and $b$-tangent bundle. We know that, if $\mathrm{p}(E)$ is the total Pontrjagin class of the real bundle $E$, then $2 \mathrm{p}(E \oplus E') = 2 \mathrm{p}(E) \smile \mathrm{p}(E')$. 

Using this relation together with \Cref{prop:klaasse-line-bundle}, we conclude that
\begin{equation}
    2 \mathrm{p}(\mathrm{T}M) = 2 \mathrm{p}(\bi{\mathrm{T}}M) \smile (1 + \mathrm{p}_1(L)).
\end{equation}
However, the first Pontrjagin class $\mathrm{p}_1(L)$ is given, by definition, by the Chern class $\mathrm{p}_1 = - \mathrm{c}_{2}(L \otimes \mathbb{C})$. Because $L \otimes \mathbb{C}$ is a complex line bundle, its second Chern class vanishes. Thus, we get the following result.

\begin{proposition} \label{prop:total-pontrjagin}
    Let $(M, Z)$ be a $b$-manifold. The total Pontrjagin class of the tangent and $b$-tangent bundle are related by $2 \mathrm{p}(\mathrm{T}M) = 2 \mathrm{p}(\bi{\mathrm{T}}M)$.
\end{proposition}

The previous computations show that the difference between the stable characteristic classes of the bundles $\mathrm{T}M$ and $\bi{\mathrm{T}} M$ are always torsion. There is a direct proof using Chern--Weil theory: because $L$ is a line bundle with locally constant transition functions, it admits a flat connection. \Cref{prop:klaasse-line-bundle} together with the additivity of the Chern--Weil morphism yields the result.

\subsection{Classes in real K-theory}

In this section, we consider the classes of the bundles $\mathrm{T} M$ and $\bi{\mathrm{T}} M$ in the K-theory ring. Remember that the set of equivalence classes of real vector bundles over $M$, denoted $\operatorname{Vect}(M)$, is an abelian semi-ring with sum $\oplus$ and product $\otimes$. The \emph{real K-theory ring}, denoted by $\mathrm{KO}(M)$, is the Grothendieck group of the abelian semi-group $\operatorname{Vect}(M)$. By definition, there exists a unique morphism of abelian semi-groups $f \colon \operatorname{Vect}(M) \to \mathrm{KO}(M)$ satisfying the following universal property: for any morphism $g \colon \operatorname{Vect}(M) \to A$ of abelian semi-groups, there exists a unique morphism $h \colon \mathrm{KO}(M) \to A$ such that $g = hf$. The reader is referred to \cite{atiyah_k-theory_1967} and \cite{bott_lectures_1969} for further details.

If we have a real vector bundle $E$, we denote by $[E] \coloneqq f(E)$ its class in the ring $\mathrm{KO}(M)$. It is known that two vector bundles $E, E' \in \operatorname{Vect}(M)$ satisfy $[E] = [E']$ if and only if there exists a bundle $F \in \operatorname{Vect}(M)$ such that $E \oplus F \cong E' \oplus F$. With this criterion, we can give a characterization of the equality $[\mathrm{T}M] = [\bi{\mathrm{T}}M]$ in terms of combinatorial data.

\begin{proposition} \label{prop:ko-theory-bman}
    Let $(M, Z)$ be a $b$-manifold. The classes $[\mathrm{T} M], [\bi{\mathrm{T}} M] \in \operatorname{KO}(M)$ are equal if and only if the associated graph $G_{M, Z}$ is two-colorable.
\end{proposition}

\begin{proof}
    The converse implication is easy to prove. By \Cref{prop:equiv-linebund-swclass}, if $G_{M, Z}$ is two-colorable, then \Cref{prop:klaasse-line-bundle} together with \Cref{prop:equiv-linebund-swclass} yields $\mathrm{T}M \oplus \mathbb{R} \cong \bi{\mathrm{T}} M \oplus \mathbb{R}$. This directly implies $[\mathrm{T} M] = [\bi{\mathrm{T}}M]$.

    To prove the direct implication, assume there exists some vector bundle $F$ such that $\mathrm{T}M \oplus F \cong \bi{\mathrm{T}}M \oplus F$. By a similar argument to that in \cref{subsec:gluing-tangent,subsec:gluing-bm-tangent}, the gluing data over open sets $U_\alpha, U_\beta$ and $V_\gamma$ is
    \begin{equation*}
        g_{\alpha \gamma}' = g_{\beta \gamma}' =  \begin{pmatrix}
            1 & & \\ & \operatorname{id_{\mathrm{T} Z}} & \\ & & \operatorname{id}_F
        \end{pmatrix},
    \end{equation*}
    for the bundle $\mathrm{T}M \oplus F$, and
    \begin{equation*}
        g_{\alpha \gamma} = \begin{pmatrix}
            1 & & \\ & \operatorname{id_{\mathrm{T} Z}} & \\ & & \operatorname{id}_F
        \end{pmatrix}, \quad g_{\beta \gamma} = \begin{pmatrix}
            - 1 & & \\ & \operatorname{id_{\mathrm{T} Z}} & \\ & & \operatorname{id}_F
        \end{pmatrix}
    \end{equation*}
    for the bundle $\bi{\mathrm{T}} M \oplus F$. A similar argument to that in \Cref{prop:comb-obstr} shows the graph $G_{M, Z}$ is two-colorable.
\end{proof}

\section{The Euler class and a Poincaré-Hopf index theorem for \texorpdfstring{$b$}{b}-tangent bundles} \label{sec:Euler-class}

In the previous sections, we computed the Stiefel--Whitney and Pontrjagin classes using the expression in equation~\eqref{eq:sw-class-btang}. This approach is valid because these characteristic classes are \emph{stable}. The Euler class is the only notable exception to this principle. Thus, a different approach is required to compute the Euler class of a $b^m$-tangent bundle.
\footnote{If we were to mimick the approach given in equation \eqref{eq:sw-class-btang}, the product formula $\chi(E \oplus F) = \chi(E) \smile \chi(F)$ together with $\chi(\mathbb{R}) = 0$ would yield the trivial identity $0 = 0$. Here we are assuming that the bundle $L$ is trivial; as we will see, this assumption is necessary.}

We begin by recalling that the Euler class of a vector bundle $E$ is defined if and only if $E$ is orientable. By \Cref{prop:equiv-linebund-swclass}, we know the line bundle $L$ is trivial and all other characteristic classes agree, modulo torsion. In the specific case of the tangent bundle, i.e. $E = \mathrm{T}M$, the Euler class can be recovered by means of the de Rham isomorphism in terms of the Euler number,
\begin{equation*}
    \chi(M) = \langle e(\mathrm{T}M), [M] \rangle = \int_M e(\mathrm{T}M).
\end{equation*}
Observe that this quantity does not depend on the choice of orientation, since once an orientation of $M$ is fixed, it canonically determines an orientation of $\mathrm{T}M$, and this choice leaves $\chi(M)$ invariant.

To prove a similar result for the Euler class $\bi{\mathrm{T}} M$, we have to discuss beforehand a rule to assign an orientation to $\bi{\mathrm{T}}M$. Following Bott and Tu~\cite{bott_differential_1982}, an orientation on an $n$-dimensional sphere bundle is a choice of open cover $\mathcal{U}$ of trivializing open sets $U_\alpha \subseteq M$ and generators $[\sigma_\alpha] \in \mathrm{H}^n(\mathbb{S}^n) \cong \mathbb{Z}$ such that the transition functions $g_{\beta \alpha}$ associated to $\mathcal{U}$ satisfy $g_{\beta \alpha}^*[\sigma_\alpha] = [\sigma_\beta]$. An orientation of a vector bundle $E$ is an orientation of its sphere bundle $\mathbb{S}E$ induced from any metric.

There is a more convenient way of describing an orientation for our purposes. In a trivializing open cover $\mathcal{U}$ of a vector bundle $E$, the local generators of $\mathrm{H}^n(\mathbb{S}^n)$ are in bijective correspondence with elements of $\mathbb{Z}_2$. Moreover, the action of a transition function $g_{\beta \alpha}$ on an element $e \in \mathbb{Z}_2$ is given by $g_{\beta \alpha}(e) = \operatorname{sign} \det g_{\beta \alpha} \cdot e $. Consequently, we can form the \emph{orientation bundle} $\mathscr{o}(E)$ as the bundle with fibre $\mathbb{Z}_2$ and transition data $h_{\beta \alpha} \coloneqq \operatorname{sign} \det g_{\beta \alpha}$. Equivalently, it can be defined as the associated bundle $\mathscr{o}(E) = \mathbb{Z}_2 \times_{\operatorname{GL}_n(\mathbb{R})} \operatorname{Fr}(E)$, where the $\operatorname{GL}_n(\mathbb{R})$-action on $\mathbb{Z}_2$ is precisely given by $f \cdot e \coloneqq \operatorname{sign} \det f \cdot e$. An orientation is simply a global section $o \in \Gamma(\mathscr{o}(E))$.

Let us consider now the open cover $\mathcal{U}$ of $M$ described in \cref{sec:isomorphism-bundles}. We can describe the orientation bundles of $\mathrm{T}M$ and $\bi{\mathrm{T}}M$ in terms of the induced gluing functions. Because $M$ is orientable, so are the open sets $U_\alpha$ and $V_\gamma$: therefore, their respective orientation bundles are trivial. The gluing functions of the orientation bundle $\mathscr{o}(\mathrm{T}M)$ are given from \cref{subsec:gluing-tangent} as $h_{\alpha \gamma}' = \operatorname{sign} \operatorname{det} g_{\alpha \gamma}' = 1$. Similarly, from \cref{subsec:gluing-bm-tangent} the gluing functions for $\mathscr{o}(\bi{\mathrm{T}}M)$ are given by $h_{\alpha \gamma} = \operatorname{sign} f|_{U_\alpha}$. In particular, if $c$ is a coloring of $G_{M, Z}$ induced from a global defining function $f$, we have $h_{\alpha \gamma} = c(U_\alpha)$.

Let us now describe a procedure to associate, given an orientation of $M$ and a coloring $c$ of the associated graph $G_{M, Z}$, an orientation of $\bi{\mathrm{T}}M$. Let us denote the induced orientation of the tangent bundle as $o' \colon M \to \mathscr{o}(\mathrm{T}M)$. Let us consider an open cover $\mathcal{U}$ as in \cref{subsec:gluing-tangent}. Then, we declare an orientation $o \colon M \to \mathscr{o}(\bi{\mathrm{T}M})$ by the following rules:
\begin{enumerate}
    \item over an open set $U_\alpha$ with $c(U_\alpha) = 1$, the anchor map $\rho \colon \bi{\mathrm{T}}M \to \mathrm{T}M$ is orientation-preserving, and
    \item over an open set $U_\beta$ with $c(U_\beta) = -1$, the anchor map $\rho \colon \bi{\mathrm{T}}M \to \mathrm{T}M$ is orientation-reversing.
\end{enumerate}

\begin{proposition} \label{prop:induced-or-bmtangent}
    In the previous considerations, the map $o \colon M \setminus Z \to \mathscr{o}(\bi{\mathrm{T}}M)$ extends to a global section.
\end{proposition}

\begin{proof}
    Recall that, in the construction of a bundle in terms of gluing data $g_{\beta\alpha}$, a collection of local sections $\{s_\alpha\}$ glues to a global section if and only if
    \[
        s_\beta\big|_{U_{\alpha\beta}} = g_{\beta\alpha}\, s_\alpha\big|_{U_{\alpha\beta}}.
    \]

    Therefore, we have to find a collection of sections $o_\gamma \in \mathscr{o}(\mathrm{T}V_\gamma)$ satisfying the equations $o_\gamma = h_{\alpha \gamma} o_\alpha = c(U_\alpha)$ and $o_\gamma = h_{\beta \gamma} o_\beta = - c(U_\beta)$. From the first equation the only condition is $c(U_\alpha) = - c(U_\beta)$ and, because $c$ is a coloring, this relation holds.
\end{proof}

The Poincaré-Hopf theorem for vector bundles relates the Euler number, i.e., the integral of the Euler class $\chi(E)$, to the indices of the zeros of a section of $E$. Our argument is an adaptation of the standard proof in Bott and Tu \cite[Theorem 11.16]{bott_differential_1982}, and therefore we briefly review the proof of the theorem in the general setting.

Let us assume that we have an orientable vector bundle $E \to M$ and a global angular form $\psi$ in the sphere bundle $\pi \colon \mathbb{S}E \to M$ (defined with respect to some metric in $E$). The Euler class measures the failure of $\psi$ to be closed by means of the formula $\diff \psi = - \pi^* e(E)$.

Assume now we have a section $s \colon M \to E$ with isolated zeros $F = \{p_1, \ldots, p_n\} \subset M$. We consider a collection of sufficiently small balls $B_{\varepsilon,i} \subseteq M$ (with regards to some metric in $M$, which is not relevant in the discussion) and define the set $M_\varepsilon = M \setminus \bigcup_{i} B_{\varepsilon,i}$. Over this set, the section $s$ induces a section of the sphere bundle $\pi \colon \mathbb{S}E \to M_\varepsilon$. The equation $\pi s = \operatorname{id}$ yields the identity $s^* \pi^* = \operatorname{id}$.

The integral of the Euler class can now be localized to $F$ after an application of Stokes' theorem:
\begin{align*}
    \int_M e(E) = \lim_{\varepsilon \to 0^+} \int_{M_\varepsilon} e(E) &= \lim_{\varepsilon \to 0^+} \int_{M_\varepsilon} s^* \pi^* e(E) \\
    &= \lim_{\varepsilon \to 0^+} \int_{M_\varepsilon} - s^* \diff \psi = \sum_{p_i \in F} \int_{\partial B_{\varepsilon,i}} s^* \psi.
\end{align*}
Assuming now that the balls $B_{\varepsilon, i}$ are sufficiently small to have isomorphisms $\mathbb{S}E|_{B_{\varepsilon, i}} \cong B_{\varepsilon,i} \times \mathbb{S}^{k - 1}$, Bott and Tu show that the last integral is given by
\begin{equation*}
    \int_{\partial B_{\varepsilon,i}} s^* \psi = \int_{\partial B_{\varepsilon,i}} s^* \operatorname{pr}_2^* \sigma = \operatorname{ind}(s, p).
\end{equation*}
Here, $\sigma$ is a generator of the cohomology $\mathrm{H}^{k - 1}(\mathbb{S}^{k - 1})$ chosen such that the isomorphism $\mathbb{S}E|_{B_{\varepsilon, i}} \cong B_{\varepsilon,i} \times \mathbb{S}^{k - 1}$ is orientation-preserving.

We are now in a position to state and prove a Poincaré--Hopf theorem for $b$-tangent bundles.

\begin{theorem} \label{thm:b-Poincare-Hopf}
    Let $(M, Z)$ be a compact $b$-manifold and assume $f \colon M \to \mathbb{R}$ is a global defining function for $Z$. Let $\rho \colon \bi{\mathrm{T}}M \to \mathrm{T}M$ be the anchor map, obtained from the natural inclusion of fields $\bi{\mathfrak{X}}(M) \xhookrightarrow{} \mathfrak{X}(M)$. Let $\operatorname{Conn}(M \setminus Z)$ be the set of connected components of $M \setminus Z$. Let $c \colon \operatorname{Conn}(M \setminus Z) \to \{+1, -1\}$ be the coloring induced by the function $f$ and endow $\bi{\mathrm{T}}M$ with the orientation described in \Cref{prop:induced-or-bmtangent}. Assume $X \in \bi{\mathfrak{X}}(M)$ is a section with isolated zeros outside $Z$. Then,
    \begin{equation} \label{eq:b-Poincare-Hopf}
        \chi(\bi{\mathrm{T}}M, c) = \int_M e(\bi{\mathrm{T}M}) = \sum_{p \in X^{-1}(0)} c(p) \, \operatorname{ind}(\rho(X), p).
    \end{equation}
    Moreover, if we assume $\operatorname{dim} M = 2k$, the following formula holds
    \begin{equation} \label{eq:comp-Poincare-Hopf}
        \chi(\bi{\mathrm{T}}M, c) = \int_M e(\bi{\mathrm{T}}M) = \sum_{U \in \operatorname{Conn}(M \setminus Z)} c(U) \, \chi(U).
    \end{equation}
\end{theorem}

\begin{remark}
    It is important to notice that, in equation \eqref{eq:b-Poincare-Hopf}, we are computing the index of the vector field $\rho(X)$, that is, of the $b$-vector field considered as an honest vector field. This result has important dynamical consequences, which we explore in \Cref{rmk:Poincare-Hopf-sign-change}.
\end{remark}

\begin{remark} \label{rmk:definition-coloring}
    In equations \eqref{eq:b-Poincare-Hopf} and \eqref{eq:comp-Poincare-Hopf} we are simultaneously denoting by $c$ two different functions: on one hand, the function $c' \colon M \setminus Z \to \{+ 1, -1\}$ which assigns a coloring to a point $p \in M \setminus Z$ and, on the other hand, the function $c'' \colon \operatorname{Conn}(M \setminus Z) \to \{+1, -1\}$ which assigns a coloring to the whole connected component $U \in M \setminus Z$. Because the $c'$ is locally constant, it automatically defines $c''$. Conversely, given $c''$ one may define $c'$ on any point $p \in M \setminus Z$ by the value of $c''$ on the unique connected component that contains it. Both constructions are clearly inverse to each other. 
\end{remark}

\begin{proof}
    We begin by proving the first equality. Let us assume we have a $b$-vector field $X \in \bi{\mathfrak{X}}(M)$ with isolated zeros $\{p_1, \ldots, p_k\} \in M \setminus Z$. We can assume, without loss of generality, that the zeros lie outside $Z$. Following the proof of Bott and Tu sketched before the statement of \Cref{thm:b-Poincare-Hopf}, we have the equality
    \begin{equation*}
        \int_M e(\bi{\mathrm{T}M}) = \sum_{p_i} \int_{\partial B_{\varepsilon, i}} s^* \psi = \int_{\partial B_{\varepsilon,i}} s^* \operatorname{pr}_2^* \sigma,
    \end{equation*}
    where we may consider $B_{\varepsilon, i}$ to be small enough to satisfy $B_{\varepsilon, i} \cap Z = \emptyset$. 
    
    Remember now that $\sigma \in \mathrm{H}^{k - 1}(\mathbb{S}^{k - 1})$ is chosen to make the isomorphism $\bi{\mathbb{S}} M|_{B_{\varepsilon, i}} \coloneqq \mathbb{S}(\bi{\mathrm{T}M})|_{B_{\varepsilon, i}} \cong B_{\varepsilon, i} \times \mathbb{S}^{k - 1}$ orientation-preserving. If we consider similarly the generator $\sigma' \in \mathrm{H}^{k - 1}(\mathbb{S}^{k - 1})$ obtained from the isomorphism $\mathbb{S}M|_{B_{\varepsilon, i}} \cong B_{\varepsilon, i} \times \mathbb{S}^{k - 1}$, we then have $\sigma = c(p) \cdot \sigma'$. Indeed, this follows from the choice of orientation in \Cref{prop:induced-or-bmtangent}. As a consequence,
    \begin{equation*}
        \operatorname{ind}(X, p_i) = \int_{\partial B_{\varepsilon, i}} s^* \operatorname{pr}_2^* \sigma = c(p_i) \int_{\partial B_{\varepsilon, i}} s^* \operatorname{pr}_2^* \sigma' = c(p_i) \operatorname{ind}(\rho(X), p_i).
    \end{equation*}
    The proof follows from the general Poincaré--Hopf theorem and the previous expression.

    Regarding the second equality, we may start with a vector field $X'$ satisfying $\mathcal{L}_{X'} f < 0$ locally around $Z$. Indeed, consider a tubular neighbourhood $\varphi \colon U \supseteq Z \to \mathbb{R} \times Z$ for which $\varphi_*f = \operatorname{pr}_1$. Denote by $t$ the coordinate induced by $\mathbb{R}$. Extend now the vector field $X' = \varphi_*^{-1}(-\partial/\partial t)$ to $M$, which satisfies $0 > \varphi^*(\mathcal{L}_{-\partial/\partial t})t = \mathcal{L}_{X'} f$. Then, under a slight perturbation, we may consider a vector field $X''$ which is transversal to the zero section of $\mathrm{T}M \to M$ and satisfies $\mathcal{L}_{X''}f < 0$.
    
    Consider now the $b$-vector field $X = f X''$, which has isolated zeros outside $Z$. Let us denote $F = X^{-1}(0) \subset M \setminus Z$. We can now apply the first part of the theorem and rearrange the sum, according to \Cref{rmk:definition-coloring}, to obtain
    \begin{align*}
        \sum_{p \in F} c(p) \, \operatorname{ind}(\rho(X), p) &= \sum_{U \in \operatorname{Conn}(M \setminus Z)} \sum_{p \in U \cap F} c(p) \, \operatorname{ind}(\rho(X), p) \\
        &= \sum_{U \in \operatorname{Conn}(M \setminus Z)} c(U) \sum_{p \in U \cap F} \operatorname{ind}(\rho(X), p).
    \end{align*}
    Now, we can normalize the $b$-vector field $\rho(X)$ to a transverse vector field to $\partial U$ without changing the zeros as section for every component $U \in \operatorname{Conn}(M \setminus Z)$. We apply now \cite[Chapter 13, Corollary 66]{spivak_comprehensive_1999} and conclude $\sum_{p \in U \cap F} \operatorname{ind}(\rho(X), p) = \chi(U)$, concluding the proof.
\end{proof}

\begin{figure}[t]
    \centering
    \begin{subfigure}[b]{0.32\textwidth}
         \centering
         \includegraphics[width=\textwidth,angle=270,origin=c]{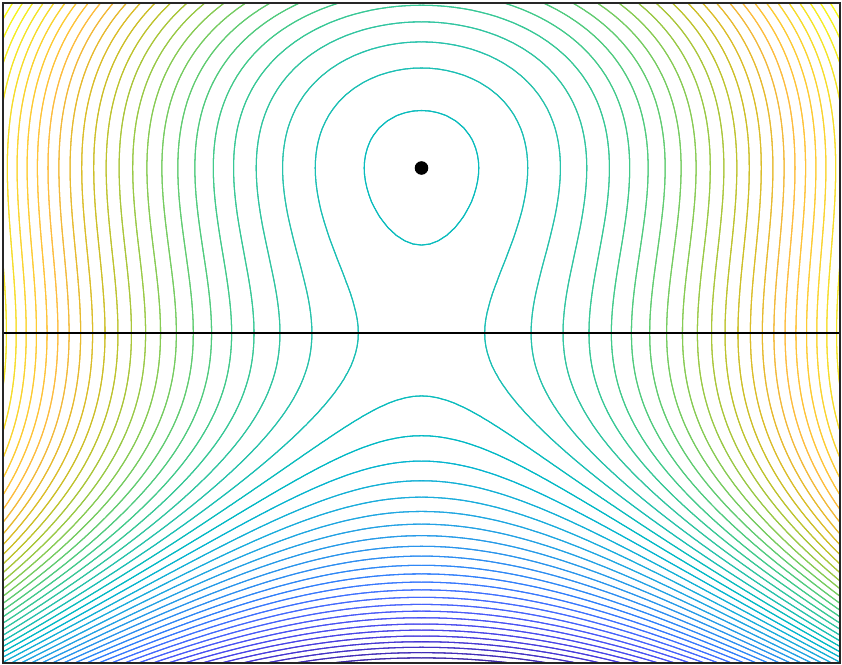}
         \caption{Plot of $f_\delta$ for $\delta > 0$.}
         \label{subfig:del+}
     \end{subfigure}
     \begin{subfigure}[b]{0.32\textwidth}
         \centering
         \includegraphics[width=\textwidth,angle=270,origin=c]{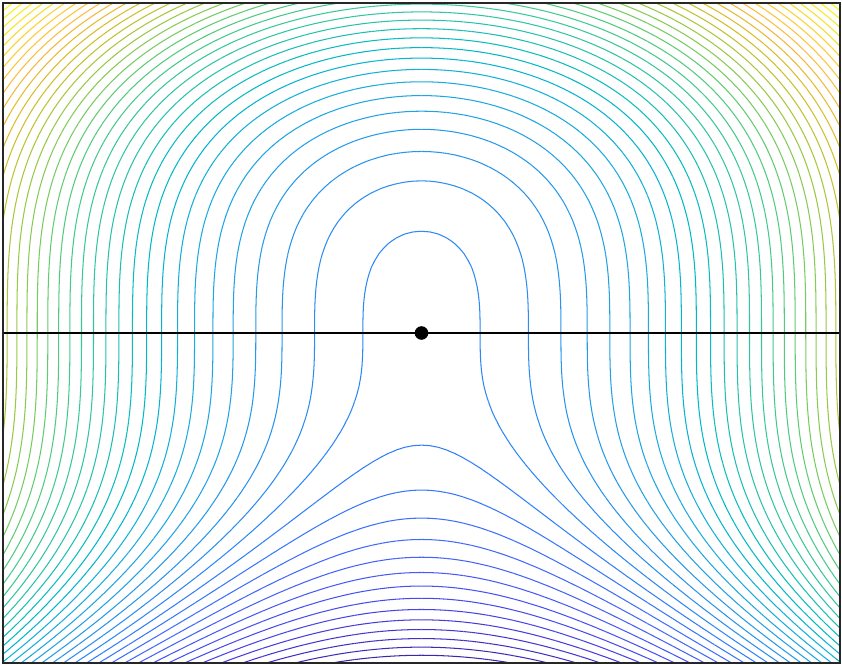}
         \caption{Plot of $f_\delta$ for $\delta = 0$.}
         \label{subfig:del0}
     \end{subfigure}
     \begin{subfigure}[b]{0.32\textwidth}
         \centering
         \includegraphics[width=\textwidth,angle=270,origin=c]{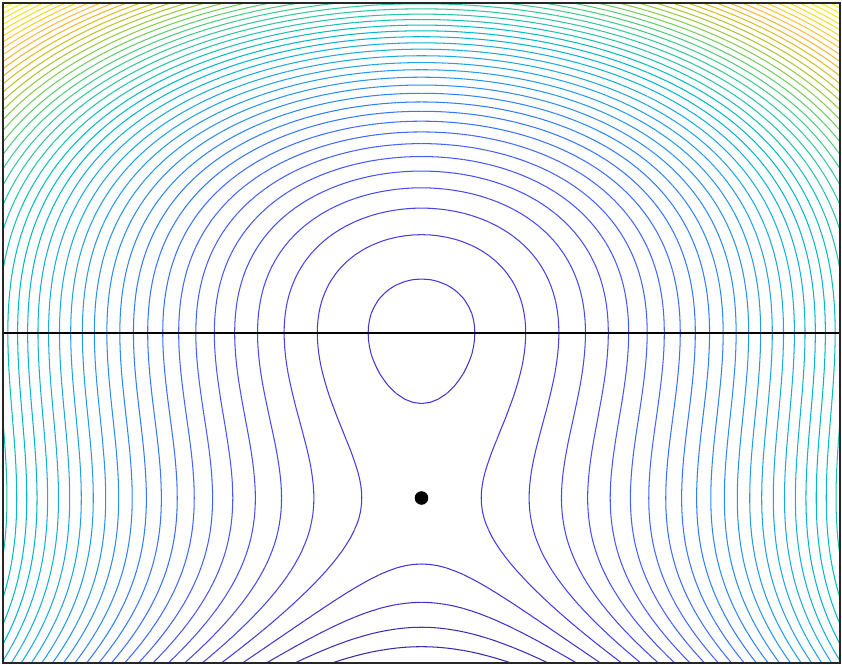}
         \caption{Plot of $f_\delta$ for $\delta < 0$.}
         \label{subfig:del-}
     \end{subfigure}
    \caption{Plots with various level sets of the function $f_\delta = x^3/3 - \delta x^2/2 + y^2/2$, whose gradient is $X_\delta$, for different values of $\delta$. The black line represents the hypersurface $Z = \{x = 0\}$, and the black point represents the zero $p_\delta = (\delta, 0)$.}
    \label{fig:contour-plots}
\end{figure}

\begin{remark} \label{rmk:Poincare-Hopf-sign-change}
    We observe that, in equation \eqref{eq:b-Poincare-Hopf}, we can perturb the $b$-vector field $X$ to transport a single zero $p \in X^{-1}(0)$ across the critical set $Z$. By the invariance of the Euler number, the index of the vector field $\rho(X)$ has to change sign: this phenomenon has implications on the dynamics of $b$-vector fields.

    We can explicitly observe this phenomenon locally. For simplicity, let us consider a $b$-manifold $(M, Z)$ with $\dim M = 2$ and a local chart $(U, \varphi)$ with coordinates $x, y$ such that $\varphi(U \cap Z) = \Set{x = 0}$. We consider the $\delta$-dependent $b$-vector field
    \begin{equation*}
        X_\delta = x (x - \delta) \frac{\partial}{\partial x} + y \frac{\partial}{\partial y},
    \end{equation*}
    whose contour plot we represent in \cref{fig:contour-plots}. The only zero of $X$ (as a $b$-vector field) is located at the point $p_\delta = (\delta, 0)$. The linearization of $X_\delta$ around the point $p_\delta$, for $\delta \neq 0$, is given by
    \begin{equation*}
        X_{\text{lin}, \delta} (x, y) = \operatorname{sgn}(\delta) |\delta| x \frac{\partial}{\partial x} + y \frac{\partial}{\partial y}.
    \end{equation*}
    \begin{itemize}
        \item If $\delta > 0$, the linearization $X_{\text{lin}, \delta}$ is a radial-like vector field, and the index is $\operatorname{ind}(\rho(X_\delta), p_\delta) = 1$. In \cref{subfig:del+} we can observe that the fixed point $p_\delta$ is indeed an elliptic singularity of $X_\delta$.
        \item Similarly, if $\delta < 0$ the linearization $X_{\text{lin}, \delta}$ has index $\operatorname{ind}(\rho(X_\delta), p_\delta) = - 1$. In \cref{subfig:del-} we can observe that the fixed point $p_\delta$ is a hyperbolic singularity of $X_\delta$.
    \end{itemize}
    We stress that the change of index is compensated in the sum \eqref{eq:b-Poincare-Hopf} by the coefficient $c(p_\delta)$ obtained from the coloring of the manifold $M$. Therefore, the sign change induced from the coloring of $M$ has a natural dynamical interpretation: it transforms elliptic fixed points into hyperbolic fixed points.
\end{remark}

\begin{remark} \label{rmk:Poincare-Hopf-zero-Z}
    In the proof of \Cref{thm:b-Poincare-Hopf}, we imposed that the zeros of $X$ should lie outside the critical set $Z$. It is possible, however, to compute the index of the $b$-vector field $X$ for a zero $p \in X^{-1}(0) \cap Z$.

    Let us consider, for simplicity,  once again, the vector field $X_0$ in \Cref{rmk:Poincare-Hopf-sign-change}, i.e.,
    \begin{equation*}
        X_0 = x^2 \frac{\partial}{\partial x} + y \frac{\partial}{\partial y},
    \end{equation*}
    whose contour plot can be found in \cref{subfig:del0}. The only zero of $X_0$ as a $b$-vector field is the point $p_0 = (0, 0)$. To compute the index $\operatorname{ind}(X_0, p_0)$ we have to recall that the $b$-tangent bundle is locally generated by the sections $\partial/\partial x$ and $\partial/\partial y$. As a consequence, $X_0$ is expressed in this basis as
    \begin{equation*}
        X_0 = x \cdot \bigg(x \frac{\partial}{\partial x}\bigg) + y \cdot \frac{\partial}{\partial y}.
    \end{equation*}
    Therefore, we see that $X_0$ is a radial vector field and thus $\operatorname{ind}(X_0, p_0) = 1$. Observe that this matches the computations in \Cref{rmk:Poincare-Hopf-sign-change}.

    If we try to compute the index of the vector field $\rho(X_0)$ at $p_0$ we get, however, $\operatorname{ind}(\rho(X_0), p_0) = 0$. Therefore, the conclusions of \Cref{thm:b-Poincare-Hopf} break down if we do not restrict ourselves to $b$-vector fields with zeros outside $Z$.
    
    This fact also explains why we did not consider this possibility in \Cref{thm:b-Poincare-Hopf}: a Poincaré–Hopf theorem does indeed hold in full generality for $\bi{\mathrm{T}}M$, but it does not provide meaningful insight into the dynamics of the field $X$.
\end{remark}

\begin{example}
    Let us consider the $b$-sphere $(\mathbb{S}^2, \mathbb{S}^1)$, where we make the identifications
    \begin{align*}
        \mathbb{S}^2 &= \Set*{ (x, y, z) \in \mathbb{R}^3 \given x^2 + y^2 + z^2 = 1}, \\
        \mathbb{S}^1 &= \mathbb{S}^2 \cap \Set{z = 0} = \Set*{ (x, y, 0) \in \mathbb{R}^3 \given x^2 + y^2 = 1 }.
    \end{align*}
    The Poincaré-Hopf formula \eqref{eq:b-Poincare-Hopf} yields, in this case, $\chi(\bi{\mathrm{T}}M, c) = \chi(\mathbb{D}^2) - \chi(\mathbb{D}^2) = 0$ for any choice of coloring $c$. This result matches \cite[Example 2.4.47]{Brugues2024}, where the author proves that $\bi{\mathrm{T}} \mathbb{S}^2$ is parallelizable. In particular, the Euler number obstructs an isomorphism $\bi{\mathrm{T}} \mathbb{S}^2 \cong \mathrm{T} \mathbb{S}^2$ given $\chi(\mathrm{T}\mathbb{S}^2) = \chi(\mathbb{S}^2) = 2$.
\end{example}

The previous computation may be generalized to any $b$-sphere $(\mathbb{S}^n, \mathbb{S}^{n - 1})$, showing $\chi(\bi{\mathrm{T}\mathbb{S}^n}) = 0$. This computation already provides an obstruction for an isomorphism $\bi{\mathrm{T}} \mathbb{S}^n \cong \mathrm{T} \mathbb{S}^n$ when the dimension is even.

\begin{corollary} \label{cor:euler-obstruction-spheres}
    Let us consider the $b$-manifold $(\mathbb{S}^n, \mathbb{S}^{n - 1})$. If $n = 2k$, we have $\bi{\mathrm{T}}\mathbb{S}^n \not\cong \mathrm{T} \mathbb{S}^n$.
\end{corollary}

\section{The isomorphism problem and a criterion for \texorpdfstring{$b$}{b}-spheres} \label{sec:iso-b-spheres}

In this section we study the isomorphism problem from a global point of view. As previously mentioned, this is an instance of a more general extension problem and, as such, explicit answers depend on the particular structure of the $b$-manifold $(M, Z)$.

\subsection{Low-dimensional cases}

There are specific cases, however, where a complete answer is possible. When $\dim M=1$, the topology of the pair $(M,Z)$ is sufficiently simple that no obstructions beyond orientability of the bundle occur. This was already observed in \Cref{prop:btang-s1}.

A closely related phenomenon appears in dimension $3$. In \cite{cardona2024existenceclassificationbcontactstructures}, Cardona and Oms show that, for a separating hypersurface $Z\subset M^3$, the only obstruction to the existence of an isomorphism $\mathrm{T}M\cong \bi{\mathrm{T}}M$ is the orientability of $\bi{\mathrm{T}}M$, which by \Cref{prop:equiv-linebund-swclass} is equivalent to the existence of a global defining function. More precisely, they prove the following.

\begin{proposition}[{Cardona, Oms~\cite[Proposition 2.4]{cardona2024existenceclassificationbcontactstructures}}]  \label{prop:cardona-oms-iso}
    Let $Z \subset M^3$ be a separating surface. Then
    \begin{enumerate}
        \item $\bi{\mathrm{T}M}$ is isomorphic to $\mathrm{T} M$,
        \item there exists a non-vanishing section of $\bi{\mathrm{T}}M$ and therefore also a $b$-plane field on $M$.
    \end{enumerate}
    In particular, in dimension $3$, $\bi{\mathrm{T}}M$ is parallelizable when $Z$ is separating.
\end{proposition}

The argument relies on the fact that, when $Z$ is separating, the bundles $\mathrm{T}M$ and $\bi{\mathrm{T}}M$ are stably isomorphic, together with a result of Eliashberg asserting that stable isomorphism implies actual isomorphism in dimensions $1$, $3$, and $7$.

\begin{proposition}[Eliashberg {\cite[Proposition~3.7]{eliashberg_singularities}}] \label{prop:eliashberg-iso}
    Let $M$ be a closed oriented $n$-manifold and let $\xi$ and $\xi'$ be oriented rank-$n$ vector bundles over $M$. If $n=1,3,$ or $7$ and $\xi$ is stably isomorphic to $\xi'$, then $\xi$ is isomorphic to $\xi'$.
\end{proposition}

Combining this with \Cref{prop:equiv-linebund-swclass} yields the following unified criterion.

\begin{theorem} \label{thm:iso-dim-1-3-7}
    Let $(M,Z)$ be a $b$-manifold with $\dim M = 1,3,$ or $7$. Then there exists an isomorphism $\mathrm{T}M\cong \bi{\mathrm{T}}M$ if and only if the associated graph $G_{M,Z}$ is two-colourable.
\end{theorem}

\subsection{$b$-Spheres as a test case}

Outside dimensions $1$, $3$, and $7$, the preceding argument no longer applies. To illustrate the genuinely non-stable nature of the problem, we now analyze a family of examples where explicit computations are possible: the $b$-spheres $(\mathbb{S}^n,\mathbb{S}^{n-1})$.

We identify $\mathbb{S}^n$ with the unit sphere in $\mathbb{R}^{n+1}$ and take the equator $\mathbb{S}^{n-1}=\mathbb{S}^n\cap\{x_n=0\}$ as the critical hypersurface. The complement consists of the northern and southern hemispheres $\mathbb{B}^\pm=\mathbb{S}^n\cap\{\pm x_n>0\}$. The associated graph $G_{\mathbb{S}^n,\mathbb{S}^{n-1}}$ has two vertices joined by a single edge and is therefore two-colourable.

As a consequence, \Cref{thm:iso-dim-1-3-7} immediately yields $\mathrm{T}\mathbb{S}^n\cong \bi{\mathrm{T}}\mathbb{S}^n$ for $n=1,3,7$, while \Cref{cor:euler-obstruction-spheres} rules out the existence of an isomorphism in even dimensions. Thus, the only remaining case is $n=2k+1$ with $n\notin\{1,3,7\}$.

\subsection{Clutching computation and parallelizability}

Rather than working directly with the extension problem for the gluing data, we compute the clutching function of the bundle $\bi{\mathrm{T}}\mathbb{S}^n$. This approach leads to a stronger structural result.

\begin{theorem} \label{thm:parallelizable-b-spheres}
    For every $n\geq 1$, the $b$-tangent bundle $\bi{\mathrm{T}}\mathbb{S}^n$ is trivializable.
\end{theorem}

The proof proceeds by explicitly trivializing $\bi{\mathrm{T}}\mathbb{S}^n$ over the closed hemispheres and computing the resulting clutching function over the equator. In contrast with the classical tangent bundle, whose clutching map detects the parallelizability of the sphere, the $b$-tangent bundle has a constant clutching function (see the appendix), and hence is trivial. This extends the explicit computation of Brugués for $\mathbb{S}^2$ \cite[Example~2.4.47]{Brugues2024}.

Since the only parallelizable spheres are $\mathbb{S}^1$, $\mathbb{S}^3$, and $\mathbb{S}^7$, we obtain the following classification.

\begin{theorem} \label{thm:criterion-iso-bspheres}
    For the $b$-sphere $(\mathbb{S}^n,\mathbb{S}^{n-1})$, there exists an isomorphism $\mathrm{T}\mathbb{S}^n\cong \bi{\mathrm{T}}\mathbb{S}^n$ if and only if $n = 1,3,$ or $7$.
\end{theorem}

\subsection{Applications to geometric structures}

The triviality of $\bi{\mathrm{T}}\mathbb{S}^n$ has immediate geometric consequences.

The next result follows from the existence of a global frame of $\bi{\mathrm{T}}^*\mathbb{S}^n$ and standard linear-algebraic constructions.

\begin{proposition}
    Let $(\mathbb{S}^n,\mathbb{S}^{n-1})$ be the $n$-dimensional $b$-sphere. If $n$ is even (resp.\ odd), then it admits an almost-symplectic (resp.\ almost contact) $b$-structure. Moreover, if $n$ is even, it also admits a $b$-almost-complex structure.
\end{proposition}

\begin{proof}
    We begin with the case $n = 2k$. By \Cref{thm:parallelizable-b-spheres}, the bundle $\bi{\mathrm{T}}\mathbb{S}^n$ is parallelisable, and hence so is $\bi{\mathrm{T}}^*\mathbb{S}^n$. Let us choose a global frame of $\bi{\mathrm{T}}^*\mathbb{S}^n$ consisting of $b$-differential one-forms $\{\alpha_1, \ldots, \alpha_k, \beta_1, \ldots, \beta_k\}$. We claim that the $b$-form
    \begin{equation*}
        \omega = \sum_{i = 1}^k \alpha_i \wedge \beta_i
    \end{equation*}
    defines an almost-symplectic $b$-form. Indeed, a direct computation shows that
    \begin{equation*}
        \frac{\omega^k}{k!} = \alpha_1 \wedge \beta_1 \wedge \cdots \wedge \alpha_k \wedge \beta_k.
    \end{equation*}
    Since the chosen forms together constitute a basis, we have $\omega^k \neq 0$ everywhere, and the claim follows.

    In the case $n = 2k + 1$, the previous result, together with the observation that $\bi{\mathrm{T}}^* \mathbb{S}^{n} \oplus {\mathbb{R}}$ is still parallelisable, completes the proof.

    Finally, we use the previous construction to define a $b$-almost-complex structure on $\mathbb{S}^{2k}$. If $\{X_1, \ldots, X_k, Y_1, \ldots, Y_k\}$ denotes the dual basis to $\{\alpha_i, \beta_i\}$, we define the endomorphism $J \colon \bi{\mathrm{T}}\mathbb{S}^{2k} \to \bi{\mathrm{T}}\mathbb{S}^{2k}$ by setting
    \begin{equation*}
        J(X_i) = Y_i, \quad J(Y_i) = - X_i.
    \end{equation*}
    By construction, $J^{2} = - \operatorname{id}$, which concludes the proof. A further computation shows that this choice makes $\{X_i, Y_i\}$ an orthonormal basis for the inner product $\langle \cdot, \cdot \rangle \coloneqq \omega(\cdot, J \cdot)$.
\end{proof}

\begin{remark}
    This example shows that cohomological obstructions alone are insufficient to detect the existence of geometric structures on $\mathrm{T}M$ from those on $\bi{\mathrm{T}}M$. In particular, although $\mathrm{w}(\bi{\mathrm{T}}\mathbb{S}^{2k})=\mathrm{w}(\mathrm{T}\mathbb{S}^{2k})$, the bundle $\bi{\mathrm{T}}\mathbb{S}^{2k}$ always admits an almost-complex structure, while $\mathrm{T}\mathbb{S}^{2k}$ does so only for $k=1,3$ (see for instance, \cite[Proposition~15.1]{borel_groupes-lie_1953}).
\end{remark}

\section{The edge tangent bundle} \label{sec:edge-tangent}

We conclude this article by briefly discussing the isomorphism classes of edge tangent bundles for different edge structures.

Research on edge structures, motivated by their relevance to pseudodifferential operators~\cite{mazzeo_elliptic} and twistor theory~\cite{fine_knots_2022}, focuses on submanifolds \( Z \) that are naturally fibered. Recall from \Cref{ex:2.6} that edge vector fields are tangent to \( Z \) and remain tangent to the fibers of the fibration.

Edge structures are a generalization of $b$-manifolds; consequently, the computations of \cref{sec:iso-b-spheres} become considerably more involved. We begin by discussing combinatorial obstructions, in analogy with \Cref{prop:comb-obstr}.

\begin{proposition}
    Let $(M, Z, \pi)$ be an edge structure, where the fibration $\pi$ has typical fibre $F$. Assume $\dim M - \dim F = 2k + 1$. Then, if ${}^e{\mathrm{T}}M \cong \mathrm{T}M$, the graph $G_{M, Z}$ is two-colorable.
\end{proposition}

\begin{proof}
    We begin by choosing a connection $\mathrm{T} Z = \operatorname{Vert} \oplus \operatorname{Hor}$. This choice yields the expression of ${}^e \mathrm{T} M$ in a tubular neighbourhood $V_\gamma$ of $Z$ as ${}^e \mathrm{T} V_\gamma \cong (V_\gamma \times \mathbb{R}) \oplus \mathrm{Hor} \oplus \mathrm{Vert}$. The gluing data in terms of this expression is given by
    \begin{equation*}
        g_{\alpha \gamma} = \operatorname{id}, \quad g_{\beta \gamma} = \begin{pmatrix}
            - 1 & & \\ & - \operatorname{id}_{\mathrm{Hor}} & \\ & & \operatorname{id}_{\mathrm{Vert}}
        \end{pmatrix}.
    \end{equation*}
    As a consequence, if $\operatorname{rank} (\mathrm{Hor}) + 1 = \operatorname{dim} M - \dim F = 2k + 1$, we have $\operatorname{sign} \det g_{\alpha \gamma} = - \operatorname{sign} \det g_{\beta \gamma}$ and the argument in \Cref{prop:comb-obstr} applies.
\end{proof}

\begin{example}
    Consider the pair $(\mathbb{S}^2, \mathbb{S}^1)$ with the edge structure $\pi = \operatorname{id} \colon \mathbb{S}^1 \to \mathbb{S}^1$. We show ${}^e {\mathrm{T}} \mathbb{S}^2 \cong \mathrm{T} \mathbb{S}^2$. In particular, together with \Cref{thm:criterion-iso-bspheres} we conclude that ${}^e{\mathrm{T}} M \not\cong \bi{\mathrm{T}} M$.

    The computation is as follows. In a tubular neighbourhood $\varphi \colon U \to \mathbb{S}^1 \times \mathbb{R}$, the gluing data is given by $g_{+0} = \operatorname{id}$ and $g_{-0} = - \operatorname{id}$. As discussed in \ref{sec:iso-b-spheres}, the obstruction to an isomorphism ${}^e{\mathrm{T}}M \cong \mathrm{T}M$ is codified in the homotopy class of the map $- \operatorname{id} \colon \mathbb{S}^1 \to \mathrm{GL}_2(\mathbb{R})$. This map is homotopic to the constant map $\operatorname{id} \colon \mathbb{S}^1 \to \mathrm{GL}_2(\mathbb{R})$ with explicit homotopy given by
    \begin{equation*}
        f(t) = \begin{pmatrix}
            - \cos (\pi t) & \sin (\pi t) \\ - \sin(\pi t) & - \cos (\pi t)
        \end{pmatrix}
    \end{equation*}
    As a consequence, ${}^e \mathrm{T} \mathbb{S}^2 \cong \mathrm{T} \mathbb{S}^2$.
\end{example}

The previous example admits further generalization. Let us consider an edge structure $(M, Z, \pi)$ where the fibration is given by $\pi \coloneqq \operatorname{id} \colon Z \to Z$. In this case, the sections of the edge tangent bundle ${}^e \mathrm{T} M$ are exactly vector fields which identically vanish at the hypersurface $Z$. We call such sections \emph{zero vector fields} and denote the space of all zero vector fields by ${}^0\mathfrak{X}(M)$. The edge tangent bundle receives the name of \emph{zero-tangent bundle} and is written ${}^0 \mathrm{T}M$.

In this setting, topological obstructions besides the colorability of the underlying graph do not play a role. To see this, assume there exists a global defining function $f \in \mathcal{C}^\infty(M)$ and consider the morphism at the level of sections
\begin{equation}
    \begin{array}{rccc}
        p \colon & \mathfrak{X}(M) & \longrightarrow & {}^0 \mathfrak{X}(M) \\
         & X & \longmapsto & f \cdot X.
    \end{array}
\end{equation}
This map is well-defined because $f$ vanishes identically at $Z$. Moreover, it is an isomorphism at stalks. Indeed, for any sufficiently small open set $U \subset M$ the equality $Y = f X$ can be inverted to get $X = Y / f$ over $U \setminus M$. Because equation holds on an open, dense subset, the claim follows from continuity. Therefore, at the level of bundles we have an isomorphism $\mathrm{T}M \cong {}^0 \mathrm{T}M$. Consequently, we have proved:

\begin{proposition}
    Let us consider the edge structure $(M, Z, \operatorname{id})$. If $Z$ is defined by a global function, then ${}^0 \mathrm{T}M \cong \mathrm{T}M$.
\end{proposition}

One can reproduce the proof from the construction of ${}^0 \mathrm{T}M$ in terms of gluing data. In an open cover $\mathcal{U} = \{U_\alpha, Z_\gamma\}$ in the assumptions of \cref{sec:isomorphism-bundles}, similar computations show the gluing data for ${}^0 \mathrm{T}M$ is given by
\begin{equation*}
    g_{\alpha \gamma} = \operatorname{id}, \qquad g_{\beta \gamma} = - \operatorname{id}.
\end{equation*}
If we assume the function $f$ is globally defined, the condition $f_\gamma \operatorname{id} = - \operatorname{id} f_\beta$ from the discussion in \cref{subsec:construction-gluing-data} is equivalent to the existence of gauge transformations $f_\beta$ such that $f_\beta = - \operatorname{id}$ in $U_\beta$. But, because $\operatorname{id}$ is always a globally defined gauge transformation, so is $- \operatorname{id}$ and the extension problem in this case is trivial.

In general, the isomorphism class of ${}^e \mathrm{T}M$ depends not only on the critical hypersurface $Z$, but also on the choice of fibration $\pi \colon Z \to N$. We have seen how the choice of a connection can be used to find explicit homotopy classes characterizing the obstructions to an isomorphism with the bundle $\mathrm{T}M$. 

\section{Concluding remarks} \label{sec:concluding}

The question of which \( b^m \)-tangent bundles are isomorphic is deeply rooted in geometric considerations involving singular symplectic structures. These structures are genuinely symplectic almost everywhere and can be viewed as sections of \( \bigwedge^2(^{b^m}\mathrm{T}^*M) \).

The singular symplectic structures associated with these bundles exhibit blow-up behavior near the singular set, characterized by logarithmic-type singularities. Notably, the existence of \( b^{2m} \)-symplectic structures implies the existence of symplectic structures, imposing several topological constraints on the underlying manifold. For instance, the second cohomology group of a compact symplectic manifold must not vanish.

In this concluding section, we explore the opposite end of the spectrum of singularities, focusing on folded symplectic forms, and examine the interplay between the isomorphism of these bundles and the geometric and dynamical properties of their sections. This perspective extends to $E$-structures associated with regular foliations, where the isomorphism of $E$-bundles integrating to equivalence of structure can be interpreted as a rigidity phenomenon. 

\subsection{Folded-cotangent bundles}

A contrasting scenario is presented by folded symplectic structures, as explored in \cite{cannas_da_silva_fold-forms_2010, CGW}, where the symplectic form fails to maintain its maximal rank along the critical set.

More concretely,

\begin{definition}
    Let $M$ be a $2n$-dimensional manifold. We say that $\omega \in \Omega^2(M)$ is folded-symplectic if 
    \begin{enumerate} 
        \item $\diff \omega=0$, 
        \item $\omega^n \pitchfork \mathcal{O}$, where $\mathcal{O}$ is the zero section of $\bigwedge^{2n}(\mathrm{T} M)$; hence $Z= \{ p \in M \mid \omega^n(p) = 0 \}$ is a codimension $1$ submanifold, and
        \item if $i \colon Z \to M$ is the inclusion map, ${i}^* \omega$ has maximal rank $2n-2$.
    \end{enumerate} 
    We say that $(M,\omega)$ is a \textbf{folded-symplectic manifold}, and we call the critical set $Z \subset M$ the folding hypersurface.
\end{definition}

For folded symplectic structures, it is not feasible to provide a similar treatment involving singular tangent bundles, as these are not well-defined. Singular cotangent bundles, however, are well-defined. As shown in \cite{hockensmith2015}, it is indeed possible to define a folded cotangent bundle.

How to do this? We define the folded cotangent bundle by considering its sections. In the definition, initial choices of additional information (a subbundle) are needed but in \cite{hockensmith2015} it is proved that the definition does not depend on the choices made. The argument goes as follows: given a folded symplectic manifold $M$ with folding hypersurface $Z$ we choose a rank-$1$ subbundle of $i_Z^* \mathrm{T}M$, $V$, such that for all $p\in Z$, the fibre $V_p$ is transverse to $\mathrm{T}_pZ$. Now consider the set of $1$-forms vanishing on $V$. For each open subset $U \subset M$ this set can be described as $ \Omega_V^1(U) \coloneqq \{ \alpha \in \Omega^1(U) \mid \alpha|_V = 0 \}.$ Observe that if $U \cap Z = \emptyset$, then this coincides with $\Omega^1(U)$.

Following \cite{hockensmith2015}, there exists a vector bundle $\mathrm{T}_V^*M$, called the \textbf{folded cotangent bundle}, of rank $n$, whose global sections are isomorphic to $\Omega^1_V(M)$. To illustrate these concepts, let us describe this bundle in local coordinates: Around a point in $Z$, there exist suitable local coordinates $(x_1, \ldots, x_{n-1})$ in $U\cap Z$ and a coordinate $t$ such that $(x_1, \ldots, x_{n-1}, t)$ are local coordinates, and locally $\mathrm{T}_V^* M$ is generated by $\diff x_1, \ldots, \diff x_{n-1}, t \diff t$. 

The construction of the folded cotangent bundle is unique up to isomorphism and does not depend on the choice of the subbundle $V$. 

This definition can be adapted for $m$-singularities using jets (for examples of $m$-folded structures consult \cite{delshamskiesenhofermiranda}). The procedures described in this article can be adapted to these singular cotangent spaces.

We could now perform the arguments \emph{mutatis mutandis} to conclude the analogues of Theorems \ref{thm:isomorphism-bm-bundles}, \ref{thm:b-Poincare-Hopf} and \ref{thm:criterion-iso-bspheres} in the folded context.

\subsection{Isomorphisms of bundles and topological, geometrical and dynamical properties of its sections}

This article addresses the natural question of whether singular bundles are isomorphic. Such isomorphisms, however, reveal profound insights into the intrinsic properties of the base space. For example, it has been observed that the isomorphism of certain singular bundles leads to the existence of additional geometric structures. Notably, $b^{2m}$-symplectic manifolds are shown to also admit symplectic structures.

Building on this perspective, we can extend this framework to other $E$-manifolds, particularly those associated with regular foliations. In these cases, the vector fields tangent to a given regular foliation constitute the sections of the $E$-tangent bundle. Investigating isomorphisms between such bundles raises compelling questions about the underlying foliations, shedding light on their rigidity properties and broader geometric implications.

In his groundbreaking work on magnetic monopoles, Dirac \cite{dirac} demonstrated that the study of singular equations in physical systems can lead to non-trivial topological structures. In our context, the implications of these isomorphisms for geometrical and dynamical systems are more limited. For example, they do not directly provide properties related to dynamics, such as the existence of periodic orbits (as in the Weinstein conjecture or the Arnold conjecture), since the isomorphism itself may not align with certain dynamical structures. Specifically, the isomorphism of Theorem \ref{thm:isomorphism-bm-bundles} relates sections of different bundles but does not necessarily preserve the geometric structures associated with these sections. For instance, while a vector field (Reeb or Hamiltonian) is associated with an additional geometric structure (contact or symplectic), the bundle isomorphism may not respect such structures.

Theorem~\ref{thm:criterion-iso-bspheres}, when applied to the constructions related to the Weinstein conjecture for singular contact manifolds in \cite{miranda_singular_2021} and \cite{fontana-mcnally_counterexample_2024}, shows that certain dynamical phenomena are purely analytical in nature. This indicates that, in some cases, the behaviour of solutions to singular equations is not determined by the topology of the underlying space. Determining the extent to which this independence holds remains an open question.

\appendix
\section{The clutching function of the $b$-tangent bundle of the sphere}

We detail below the computation of the clutching function of the $b$-tangent bundle of the sphere. More concretely we prove the following:

\begin{proposition}\label{prop:clutching}
    The $b$-tangent bundle of a $b$-sphere $(\mathbb{S}^n, \mathbb{S}^{n-1})$ has a constant clutching function.
\end{proposition}

\begin{proof}
    Let us begin the proof by recalling the description of the clutching function of $\mathrm{T}\mathbb{S}^n$ in \cite[Chapter 8, Section 9]{husemoller-fibrebundles}. We work in ambient coordinates and denote the north and south poles respectively by $p_{\mathrm{n}} = (0, \ldots, 0, 1)$ and $p_{\mathrm{s}} = - p_{\mathrm{n}}$. Following \cite{husemoller-fibrebundles}, we trivialize the tangent bundles of $\overline{\mathbb{B}^+}$ and $\overline{\mathbb{B}^-}$ by means of the vector bundle isomorphisms (in ambient coordinates).
    \begin{equation*}
        \begin{array}{rccc}
            \tau^+ \colon & \overline{\mathbb{B}^+} \times \mathrm{T}_{p_\mathrm{n}} \mathbb{S}^n & \longrightarrow & \mathrm{T} \overline{\mathbb{B}^+} \\
             & (p, v) & \longmapsto & R(e_n, p) \cdot v
        \end{array}
    \end{equation*}
    
    \begin{equation*}
        \begin{array}{rccc}
            \tau^- \colon & \overline{\mathbb{B}^-} \times \mathrm{T}_{p_\mathrm{s}} \mathbb{S}^n & \longrightarrow & \mathrm{T} \overline{\mathbb{B}^-} \\
             & (p, v) & \longmapsto & R(e_{n - 1}, e_n)^2 R(R(e_{n - 1}, e_n)^2 p, e_n) \cdot v.
        \end{array}
    \end{equation*}
    Here, $R(x, y)$ for $x, y \in \mathbb{S}^{n}$ denotes the rotation from $y$ to $x$ which leaves invariant the subspace $\langle x, y \rangle^\perp$.
    
    Recall from \Cref{subsec:gluing-tangent} that the gluing data for $\mathrm{T}\mathbb{S}^n$ over $\overline{\mathbb{B}^+} \cap \overline{\mathbb{B}^-} = Z = \mathbb{S}^{n-1}$ is given by the map $\operatorname{id} \colon i_Z^* \mathrm{T}\mathbb{S}^n \to i_Z^* \mathrm{T}\mathbb{S}^n$. This data can be expressed in terms of the trivializations $\tau^+$ and $\tau^-$ via the commutative diagram
    \[\begin{tikzcd}
    	{i_Z^* \mathrm{T} \mathbb{S}^n} & {i_Z^* \mathrm{T} \mathbb{S}^n} \\
    	{\mathbb{S}^{n - 1} \times \mathbb{R}^n} & {\mathbb{S}^{n - 1} \times \mathbb{R}^n}
    	\arrow["{\operatorname{id}}", from=1-1, to=1-2]
    	\arrow["{\tau^-}", from=2-1, to=1-1]
    	\arrow["{c'}"', from=2-1, to=2-2]
    	\arrow["{\tau^+}"', from=2-2, to=1-2]
    \end{tikzcd}\]
    For notational convenience we have identified the map $(\tau^+)^{-1} \tau^-(p, v) = (p, g(p) \cdot v)$ with the section $g \colon \mathbb{S}^{n-1} \to \operatorname{O}(n)$. In \cite[Chapter 8, Theorem 9.5]{husemoller-fibrebundles} it is proved that $c' = (\tau^+ )^{-1} \tau^- = R(p, e_{n - 1})^2$.
    
    We give an analogous description of the clutching function $g \colon i^*_Z \mathrm{T} \mathbb{S}^n \to i_Z^* \mathrm{T}\mathbb{S}^n$ of the $b$-tangent bundle $\bi{\mathrm{T}}\mathbb{S}^n$ which, in ambient coordinates, is simply given by the map $g = \rho_{\langle e_n \rangle^\perp}$. For $p \in \mathbb{S}^n$, we denote by $\rho_{\langle p \rangle^\perp}$ the orthogonal reflection with respect to the hyperplane $\langle p \rangle^\perp$. The description of $g$ in the choice of trivializations $\tau^+$ and $\tau^-$ is the unique map $c \colon \mathbb{S}^{n-1} \to \mathrm{O}(n)$ making the following diagram commute:
    \[\begin{tikzcd}
    	{i_Z^* \mathrm{T} \mathbb{S}^n} & {i_Z^* \mathrm{T} \mathbb{S}^n} \\
    	{\mathbb{S}^{n - 1} \times \mathbb{R}^n} & {\mathbb{S}^{n - 1} \times \mathbb{R}^n}
    	\arrow["{\rho_{\langle e_n\rangle^\perp}}", from=1-1, to=1-2]
    	\arrow["{\tau^-}", from=2-1, to=1-1]
    	\arrow["c"', from=2-1, to=2-2]
    	\arrow["{\tau^+}"', from=2-2, to=1-2]
    \end{tikzcd}\]
    
    \begin{proposition}
        In the previous assumptions, we have
        \begin{equation}
            c = (\tau^+)^{-1} \rho_{\langle e_n \rangle^\perp} \tau^- = \rho_{\langle e_{n - 1} \rangle^\perp}.
        \end{equation}
    \end{proposition}
    
    \begin{proof}
        The proof is a concatenation of the following simple facts. First, in \cite[Chapter 8, Proposition 9.8]{husemoller-fibrebundles} it is proved that $R(x, y)^2 = \rho_{\langle x \rangle^\perp} \rho_{\langle y \rangle^\perp}$. This, together with the description of the clutching function for $\mathrm{T}\mathbb{S}^n$ in the trivializations $\tau^+$ and $\tau^-$ implies $\tau^- = \tau^+ \rho_{\langle x \rangle^\perp} \rho_{\langle e_{n - 1} \rangle^\perp}$. Second, the commutation relation $R(x, e_n)^{-1} \rho_{\langle e_n \rangle^\perp} = \rho_{\langle e_n \rangle^\perp} R(x, e_n)$ holds.\footnote{
            We can prove this fact in the two-dimensional case by choosing a basis such that $e_n = (1, 0)$ and $x = (\cos \theta, \sin \theta)$: then, the result boils down to the verification that
            \begin{equation*}
                \begin{pmatrix}
                    -1 & 0 \\ 0 & 1
                \end{pmatrix} \begin{pmatrix}
                    \cos \theta & -\sin \theta \\ \sin \theta & \cos \theta
                \end{pmatrix} \begin{pmatrix}
                    -1 & 0 \\ 0 & 1
                \end{pmatrix} = \begin{pmatrix}
                    \cos \theta & \sin \theta \\ - \sin \theta & \cos \theta
                \end{pmatrix}.
            \end{equation*}
            For the higher-dimensional case, the result is reduced to the previous computation by noting that both $R(x, e_n)$ and $\rho_{\langle e_n \rangle^\perp}$ act trivially on the subspace $\langle x, e_n \rangle^\perp$.
        }
        In particular, $(\tau^+)^{-1} \rho_{\langle e_n \rangle^\perp} = \rho_{\langle e_n \rangle^\perp} \tau^+$. Therefore, we have
        \begin{align*}
            (\tau^+)^{-1} \rho_{\langle e_n \rangle^\perp} \tau^- &= \rho_{\langle e_n \rangle^\perp} \tau^+ \tau^- \\
            &= \rho_{\langle e_n \rangle^\perp} (\tau^+)^2 (\tau^+)^{-1} \tau^- \\
            &= \rho_{\langle e_n \rangle^\perp} R(x, e_n)^2 R(x, e_{n - 1})^2 \\
            &= \rho_{\langle e_n \rangle^\perp} \rho_{\langle x \rangle^\perp} \rho_{\langle e_n \rangle^\perp} \rho_{\langle x \rangle^\perp} \rho_{\langle e_{n - 1} \rangle^\perp}.
        \end{align*}
        The key observation now is that, because $\langle x, e_n \rangle = 0$ as $x \in Z = \mathbb{S}^{n} \cap \{ x_n = 0 \}$, we have that $\rho_{\langle x \rangle^\perp} \rho_{\langle e_n \rangle^\perp} = \rho_{\langle e_n \rangle^\perp} \rho_{\langle x \rangle^\perp}$ holds.\footnote{
            For a simple and coordinate-free proof, write $\rho_{\langle x \rangle^\perp} (y) = y - 2 \langle x, y \rangle x$ and observe that
            \begin{align*}
                \rho_{\langle x \rangle^\perp} \rho_{\langle e_n \rangle^\perp}(y) &= y - 2 \langle e_n, y \rangle e_n - 2 \big\langle x, y - 2 \langle e_n, y \rangle e_n \big\rangle x \\
                &= y - 2 \langle e_n, y \rangle e_n - 2 \langle x, y \rangle x + 4 \langle e_n, y \rangle \langle x, e_n \rangle x \\
                &= y - 2 \langle e_n, y \rangle e_n - 2 \langle x, y \rangle x.
            \end{align*}
            By the same computation, we get the claim.
        }
        Now, with this at hand and using the fact that reflections are nilpotent, we get
        \begin{align*}
            (\tau^+)^{-1} \rho_{\langle e_n \rangle^\perp} \tau^- &= \rho_{\langle e_n \rangle^\perp} \rho_{\langle x \rangle^\perp} \rho_{\langle e_n \rangle^\perp} \rho_{\langle x \rangle^\perp} \rho_{\langle e_{n - 1} \rangle^\perp} \\
            &= \rho_{\langle e_n \rangle^\perp}^2 \rho_{\langle x \rangle^\perp}^2 \rho_{\langle e_{n - 1} \rangle^\perp} \\
            &= \rho_{\langle e_{n - 1} \rangle^\perp} \qedhere
        \end{align*}
    \end{proof}
    
    Notice that this computation implies the clutching function $c \colon \mathbb{S}^{n - 1} \to \operatorname{O}(n)$ is constant. This concludes the proof of \Cref{prop:clutching} and thus also the proof of \Cref{thm:parallelizable-b-spheres}.
\end{proof}

\bibliographystyle{alpha}
\bibliography{isomorphism-bundles-bibliography}

\end{document}